\documentclass[reqno,10pt]{amsart}
\usepackage{amscd,amssymb,amsmath,latexsym,color,url}

\usepackage[colorlinks=true]{hyperref}
\usepackage[all]{xy}

\def\into{\hookrightarrow}

\def\toisom{\widetilde{\to}}

\def\.{,\dotsc ,}
\def\:{\colon}
\def\wt{\widetilde}

\def\ol{\overline}

\def\Val{{\rm Val}}

\def\Val{{\rm Val}}

\def\Red{{\rm Red}}

\def\Spa{{\rm Spa}}

\def\Spec{{\rm Spec}}

\def\Frac{{\rm Frac}}
\def\bfSpec{{\bf Spec}}

\def\red{{\rm red}}
\def\aff{{\rm aff}}

\def\rig{{\rm rig}}
\def\ad{{\rm ad}}

\def\an{{\rm an}}

\def\RZ{{\rm RZ}}

\def\Id{{\rm Id}}

\def\bir{{\rm bir}}

\def\bfP{{\bf P}}

\def\bfS{{\bf S}}
\def\bfT{{\bf T}}
\def\bfU{{\bf U}}
\def\bfV{{\bf V}}
\def\bfW{{\bf W}}
\def\bfX{{\bf X}}
\def\bfY{{\bf Y}}
\def\bfZ{{\bf Z}}

\def\bff{{\bf f}}
\def\bfg{{\bf g}}
\def\bfh{{\bf h}}
\def\bfi{{\bf i}}

\def\bfr{{\bf r}}

\def\bfu{{\bf u}}
\def\bfv{{\bf v}}
\def\bfw{{\bf w}}

\def\bfy{{\bf y}}

\def\gtU{{\mathfrak U}}

\def\gtW{{\mathfrak W}}
\def\gtX{{\mathfrak X}}

\def\gtf{{\mathfrak f}}

\def\gth{{\mathfrak h}}

\def\bbA{{\mathbb A}}

\def\bbG{{\mathbb G}}

\def\bbP{{\mathbb P}}

\def\calF{{\mathcal F}}

\def\calI{{\mathcal I}}
\def\calJ{{\mathcal J}}

\def\calO{{\mathcal O}}
\def\calP{{\mathcal P}}

\def\calS{{\mathcal S}}
\def\calT{{\mathcal T}}
\def\calU{{\mathcal U}}
\def\calV{{\mathcal V}}
\def\calW{{\mathcal W}}
\def\calX{{\mathcal X}}
\def\calY{{\mathcal Y}}
\def\calZ{{\mathcal Z}}

\def\oY{{\ol Y}}

\def\of{{\ol f}}
\def\og{{\ol g}}

\def\tilB{{\wt B}}

\def\tilK{{\wt K}}

\def\tilR{{\wt R}}

\def\tilT{{\wt T}}
\def\tilU{{\wt U}}

\def\tilW{{\wt W}}
\def\tilX{{\wt X}}

\def\tilZ{{\wt Z}}

\def\tila{{\wt a}}

\def\tilg{{\wt g}}
\def\tilh{{\wt h}}

\def\obfX{{\ol\bfX}}
\def\obfY{{\ol\bfY}}

\def\tilcalI{{\wt\calI}}
\def\tilcalJ{{\wt\calJ}}

\def\tilcalT{{\wt\calT}}

\def\tilcalX{{\wt\calX}}

\def\kcirc{k^\circ}

\def\oalpha{{\ol\alpha}}

\def\opi{{\ol\pi}}

\def\alp{{\alpha}}

\def\veps{\varepsilon}

\def\R+*{{\bf R^*_+}}

\def\colim{{\rm{colim}}}

\newtheorem{theor}[subsubsection]{Theorem}
\newtheorem{prop}[subsubsection]{Proposition}
\newtheorem{lem}[subsubsection]{Lemma}
\newtheorem{cor}[subsubsection]{Corollary}

\theoremstyle{definition}

\newtheorem{defin}[subsubsection]{Definition}
\newtheorem{rem}[subsubsection]{Remark}

\newtheorem{exam}[subsubsection]{Example}

\def\aff{\mathfrak{Aff}}

\def\colim{{\rm colim}}

\begin{document}

\author{Michael Temkin, Ilya Tyomkin}
\thanks{Both authors were supported by the Israel Science Foundation (grant No. 1018/11). The second author was also partially supported by the European FP7 IRG grant 248826.}
\title[On relative birational geometry and {N}agata's compactification]{On relative birational geometry and {N}agata's compactification for algebraic spaces}

\address{Einstein Institute of Mathematics, The Hebrew University of Jerusalem, Giv'at Ram, Jerusalem, 91904, Israel}
\email{temkin@math.huji.ac.il}
\address{Department of Mathematics, Ben-Gurion University of the Negev, P.O.Box 653, Be'er Sheva, 84105, Israel}
\email{tyomkin@math.bgu.ac.il}

\begin{abstract}
In \cite{temrz}, the first author introduced (relative) Riemann-Zariski spaces corresponding to a morphism of schemes and established their basic properties. In this paper we clarify that theory and extend it to morphisms between algebraic spaces. As an application, a new proof of Nagata's compactification theorem for algebraic spaces is obtained.
\end{abstract}

\maketitle

\section{Introduction}
In \cite{temrz}, the first author introduced relative Riemann-Zariski spaces associated to a separated morphism of schemes, established their basic properties, and obtained several applications such as a strong version of stable modification theorem for relative curves, and a theorem about factorization of separated morphisms, which generalizes Nagata's compactification theorem \cite{Nagata}. The aim of the current paper is to extend the methods and the results of \cite{temrz} to algebraic spaces, and thereby, prepare the ground for further generalizations, e.g., the equivariant case and the case of stacks.

In this paper, we introduce the language of {\it models}, which makes the method more intuitive. Its choice is motivated by relations between classical birational geometry, Raynaud's theory, and, what we call, relative birational geometry. The results in this paper include (i) a new proof of Nagata's compactification theorem for algebraic spaces and of a stronger result on factorization of separated morphisms of algebraic spaces, (ii) a valuative description of Riemann-Zariski spaces associated to a separated morphism of algebraic spaces, (iii) a valuative criterion for schematization, and (iv) an application of our theory to the study of Pr\"ufer morphisms and Pr\"ufer pairs of qcqs algebraic spaces.

\subsection{Relative birational geometry}\label{relbirsubsec}
Relative birational geometry is the central topic of the current paper. In this section, we explain the concept of relative birational geometry, the motivation for its study, and the relations with the classical birational and non-archimedean geometries.

\subsubsection{The absolute case}
Classical birational geometry studies algebraic varieties up-to birational equivalence, and, in particular, addresses questions about (i) birational invariants, and (ii) existence of ``good" representatives of birational classes. More precisely, birational geometry studies the localization procedure of the category of integral schemes with dominant morphisms with respect to proper birational morphisms, and the birational problems can be roughly divided into two classes: the study of the localized category, including (i), and the study of the fibers of the localization, including (ii), weak factorization problem, Chow's Lemma, etc.

\subsubsection{Birational geometry of pairs}
Often, one is only interested in modifications of a scheme that preserve a certain (pro-)open subscheme. For example, in the desingularization problems, one usually wants to preserve the regular locus of $X$. In this case, it is more natural to consider the category of pairs $(X,U)$, where $U\subseteq X$ is pro-open, quasi-compact, and schematically dense; and to study the localization functor that inverts {\em $U$-modifications}, i.e., morphisms $(X',U')\to(X,U)$ such that $X'\to X$ is proper and $U'\toisom U$. We call the latter study {\em birational geometry of pairs}, and the classical birational geometry is obtained when $U$ is a point.

\subsubsection{The general relative case}
It makes sense to study birational geometry of morphisms $U\to X$ that are not necessarily pro-open immersions. For this aim, we introduce the language of models (see \S\ref{analogysec} for some analogies explaining the terminology). A {\em model} is a triple $\bfX=(\calX, X, \psi_\bfX)$ such that $X$ and $\calX$ are quasi-compact and quasi-separated (qcqs) algebraic spaces, and $\psi_\bfX\colon\calX\to X$ is a separated schematically dominant morphism. If $\psi_\bfX$ is an open immersion (resp. affine morphism) then we say that $\bfX$ is {\em good} (resp. affine). In particular, a good model can be viewed as a pair $(X,\calX)$. A {\em morphism of models} $\bff\colon\bfY\to \bfX$ is a pair of morphisms $\gtf\colon \calY\to \calX$ and $f\colon Y\to X$ such that $\psi_\bfX\circ \gtf=f\circ \psi_\bfY$. If $\gtf$ is an isomorphism and $f$ is proper then we say that $\bff$ is a {\em modification}. The {\em relative birational geometry} studies localization of the category of models with respect to the family of modifications.

\subsubsection{Factorization and Nagata's compactification theorems}\label{factorsubsec}
Nagata's compactification theorem is a classical result first proved by Nagata \cite{Nagata} in 1963. Since then few other proofs were found, see \cite{con, Lut, temrz}. Recently, it has been generalized to the case of algebraic spaces by Conrad, Lieblich, and Olsson following some ideas of Gabber \cite{CLO}. Note also that in a work in progress \cite{Rydh}, Rydh extends the compactification theorem to tame morphisms between Deligne-Mumford stacks.

In our study of models up to modifications, one of the main results is Theorem~\ref{thm:modmainth} asserting, in particular, that any model $\bfX$ possesses a modification $\bfY$ such that $\psi_\bfY$ is affine. This is equivalent to the {\em Factorization theorem} asserting that $\psi_\bfX\:\calX\to X$ is a composition of an affine morphism $\psi_\bfY\:\calX=\calY\to Y$ and a proper morphism $Y\to X$. Since any affine morphism of finite type is quasi-projective, if $\bfX$ is of finite type then $\bfY$ above can be chosen to be good. In particular, $Y$ is an $X$-compactification of $\calX=\calY$, and we obtain a new proof of Nagata's compactification theorem for algebraic spaces. This proves {\em a posteriori} that the birational geometry of pairs contains the birational geometry of finite type models - a deep fact equivalent to Nagata's compactification.

\subsubsection{Formal birational geometry}
Our discussion of birational geometry would not be complete without the formal case. Let $k$ be a complete height-one valued field with ring of integers $\kcirc$. To any {\em admissible} (i.e., finitely presented and flat) formal $\kcirc$-scheme $\gtX$ one can associate a non-archimedean analytic space $\gtX_\eta$ over $k$ called the generic fiber. There are three theories of such analytic spaces: Tate's rigid spaces, Berkovich's $k$-analytic spaces, and Huber's adic spaces, and $\gtX_\eta$ can be defined in each of them. Any {\em admissible blow up}, i.e., a formal blow up along an open ideal, induces an isomorphism of generic fibers. Raynaud proved that the category of qcqs rigid spaces over $k$ is the localization of the category of admissible formal $\kcirc$-schemes with respect to the family of all admissible blow ups (see \cite{BL}). A part of this result is the following version of Chow's lemma: admissible blow ups of $\gtX$ are cofinal among all {\em admissible modifications}, i.e., proper morphisms $\gtX'\to\gtX$ inducing an isomorphism of the generic fibers.  In particular, Raynaud's theory provides a way to treat analytic geometry over $k$ as the birational geometry of admissible formal $\kcirc$-scheme. This approach had major contributions to the non-archimedean geometry.

\subsubsection{Analogies}\label{analogysec}
In classical birational geometry, a (proper) $k$-model of a finitely generated $k$-field $K$ is a representative of the birational equivalence class of $K$. Similarly, in formal geometry one says that $\gtX$ is a {\em formal model} of the non-archimedean space $\gtX_\eta$. In the same way, one can view a model $\bfX$ as a representative of the relative birational equivalence class defined by $\bfX$. The analogy with the formal case is especially fruitful, and some of our arguments are motivated by their analogues in Raynaud's theory (see also \cite{temrz}). Note, however, that the analogy becomes tight only when addressing good models.

\subsection{Birational spaces}\label{birspsec}
In order to study the fibers of the localization functor, it is useful to have nice invariants associated to the fibers. In this paper we consider several {\em relative birational spaces} - topological spaces associated to a model in a functorial way that depend on the model only up to modifications. In the absolute case the construction of such spaces goes back to Zariski.

\subsubsection{The classical case}
For a finitely generated field extension $K/k$ Zariski defined the {\em Riemann manifold} $\RZ(K/k)$ to be the set of all valuation rings $k\subseteq \calO\subseteq K$ such that $\Frac(\calO)=K$, see \cite[A.II.5]{Zar}. Zariski equipped $\RZ(K/k)$ with a natural quasi-compact topology, and proved that it is homeomorphic to the filtered limit of all projective $k$-models of $K$.

\subsubsection{The formal case}
A well known fact that appeared first in a letter of Deligne says that if $\gtX$ is an admissible formal scheme then the filtered limit of all its admissible modifications (or blow ups) is homeomorphic to the adic generic fiber $\gtX_\eta^\ad$ (and other generic fibers $\gtX_\eta^\an$ and $\gtX_\eta^\rig$ are its natural subsets). In particular, any finite type adic space over $k$ is homeomorphic to the filtered limit of all its formal models.

\subsubsection{The absolute Riemann-Zariski spaces}
Zariski's construction easily extends from projective varieties to arbitrary integral schemes, see \cite[\S3.2]{temst}. The {\em Riemann-Zariski space} of such a scheme $X$ is the filtered limit of the family of all modifications of $X$ in the category of locally ringed topological spaces. In particular, $\RZ(X)$ is a birational space, i.e., it only depends on $X$ up to modifications. Such spaces play an important role in desingularization problems, in some proofs of Nagata's compactification theorem (including Nagata's original proof), and in other problems of birational geometry.

Similarly to Zariski's definition, there exists a valuative description of points of $\RZ(X)$. Let $\Val(X)$ be the set of pairs consisting of a valuation ring $\calO$ of $k(X)$ and a morphism $\Spec(\calO)\to X$ extending the isomorphism of the generic points. It admits a natural structure of a locally ringed space. The valuative criterion of properness gives rise to a natural map $\red_X\:\Val(X)\to\RZ(X)$, which turns out to be an isomorphism.

\subsubsection{The relative Riemann-Zariski spaces}
The above absolute construction was extended in \cite[\S3.3]{temst} and \cite{temrz} to morphisms of schemes, i.e., to models in the category of schemes, and in this paper we generalize it further to models of algebraic spaces. We define $\RZ(\bfX)$ to be the filtered limit of the topological spaces $|X_\alp|$ over the family of all modifications $\bfX_\alp\to\bfX$. Unlike the scheme case in \cite{temrz}, we do not equip $\RZ(\bfX)$ with any further structure, see Remark~\ref{rem:RZextrastucture}. Obviously, $\RZ(\bfX)$ is a relative birational space, as it depends on $\bfX$ only up to a modification.

The situation with valuative interpretation is subtler. We introduced in \cite[\S5]{pruf} certain valuative diagrams in the category of algebraic spaces that generalize Zariski's valuations $\Spec(\calO)\to X$, and used them to refine the usual valuative criterion of properness for algebraic spaces. Such diagrams are called {\em semivaluations} in this paper (see \S\ref{semivalsec}), and the set of all semivaluations of $\bfX$ is denoted by $\Spa(\bfX)$. It admits two natural topologies, called {\em Zariski topology} and {\em constructible topology}. The space $\Spa(\bfX)$ is easy to work with, but, unfortunately, it is much bigger than $\RZ(\bfX)$, and hence only plays a technical role. However, there exists a class of ``nice" semivaluations, called {\em adic} (see \S\ref{adicvalsec}), and we denote the subspace of adic semivaluations by $\Val(\bfX)$. Although, the space $\Val(\bfX)$ is more difficult to approach, it turns out to be the ``right" space to consider.

It follows from the valuative criteria of \cite{pruf} that both $\Spa(\bfX)$ and $\Val(\bfX)$ depend on $\bfX$ only up to modifications. The spaces $\Val(\bfX)$ play central role in our study of birational geometry of models. In particular, we will often work locally around points of $\Val(\bfX)$. As in the classical case, there exist natural restriction maps $\red_\bfX\:\Val(\bfX)\to\RZ(\bfX)$ and $\Red_\bfX\:\Spa(\bfX)\to\RZ(\bfX)$, and we prove in Theorem~\ref{rzth} that $\red_\bfX$ is a homeomorphism.

\subsubsection{Categorifications}\label{categsec}
As a future project, it would be desirable to provide relative birational spaces with enough structure in order to obtain a geometric realization of birational geometry. This should lead to an independent definition of the category of birational spaces, equivalent to the category of models localized by the modifications. For comparison we note that in the formal case, this program is fully realized by (one of) the definitions of non-archimedean spaces and Raynaud's theory, in the classical case, this is done by the definition of the category $\bir_k$ in \cite[\S1]{temlocal} and \cite[Prop. 1.4]{temlocal}, and a geometric realization of birational geometry of schemes is constructed by U. Brezner in \cite{Br}.

\subsection{Overview of the paper}

\subsubsection{Models}
In Section~\ref{modelsec}, we introduce the category of models of algebraic spaces, and study basic properties of models. In particular, we define pseudo-modifications, quasi-modifications, modifications, and blow ups of models; and prove that they are preserved by compositions and base changes, and form filtered families. Here we only note that $\bff\:\bfY\to\bfX$ is a pseudo-modification if $\gtf$ is an open immersion and $f$ is separated and of finite type, and it is a quasi-modification if it is also {\em adic}, i.e., the immersion $\calY\into\calX\times_XY$ is closed. Quasi-modifications provide the right tool to localize modifications, and our main task will be to construct large enough families of pseudo-modifications, quasi-modifications, and modifications.

\subsubsection{Relative birational spaces}
In Section~\ref{relbirsec}, we associate to a model $\bfX$ the birational spaces $\RZ(\bfX)$, $\Spa(\bfX)$, and $\Val(\bfX)$. Then we show that these spaces are quasi-compact, $T_0$, and construct natural maps whose continuity, excluding $r_\bfX$, is easily established:
$$
\xymatrix{
\Spa(\bfX)\ar@<1ex>[d]^{r_\bfX}\ar[dr]^{\Red_\bfX} & \\
\Val(\bfX)\ar@{^{(}->}@<1ex>[u]^{\iota_\bfX}\ar[r]^{\red_\bfX} & \RZ(\bfX)\\
}
$$

Note that the topology of $\RZ(\bfX)$ is generated by the subsets $\RZ(\bfY)$ for compactifiable quasi-modifications $\bfY\to\bfX$, i.e., $\bfY$ is an open submodel of a modification of $\bfX$, while the topology of $\Spa(\bfX)$ is generated by the subsets $\Spa(\bfY)$ for pseudo-modifications $\bfY\to\bfX$. In particular, the proof of the $T_0$ property in Theorem~\ref{thm:ValT0} reduces to a construction of a pseudo-modification that separates semivaluations. The construction involves two pushouts (affine and Ferrand) and an approximation argument. The reason that we have to use the affine pushout is quite technical: in general there is no localization of an algebraic space at a point.

At this stage we provide $\Val(\bfX)$ with the topology induced from $\Spa(\bfX)$. In fact, the natural topology on $\Val(\bfX)$ is generated by quasi-modifications because $\Val(\bfY)\subseteq\Val(\bfX)$ only when $\bfY$ is a quasi-modification of $\bfX$ (Lemma~\ref{lem1}), but this topology is useless at this stage as we do not even know yet that there are enough quasi-modifications to distinguish points of $\Val(\bfX)$. To begin with, we only need to use quasi-compactness of $\Val(\bfX)$, so it suffices to work with the (a priori) stronger pseudo-modification topology.

\subsubsection{The family of modifications}
Theorem~\ref{thm:modmainth} is our main result about the family of modifications of a model. It asserts that affine models are cofinal among all modifications of $\bfX$, and if $\bfX$ is good, then the family of blow ups of $\bfX$ is cofinal among all its modifications. As was noted in \S\ref{factorsubsec}, this implies Factorization and Nagata's compactification theorems. Our proof of Theorem~\ref{thm:modmainth} consists of the following steps: First, we attack the problem locally, and prove in Theorem~\ref{strictqmodth} that if $\psi_\bfX$ is of finite type then any adic semivaluation $\bff\:\bfT\to\bfX$ factors through a good quasi-modification $\bfY\to\bfX$. The construction of $\bfY$ is the most technical argument in this paper. It is analogous to the construction of a pseudo-modification in Theorem~\ref{thm:ValT0} but more involved because we have to work with $\bfX$-adic models throughout the construction. In addition, it only applies to uniformizable semivaluations, thus forcing us to assume that $\bfX$ is of finite type. Then we prove in Lemma~\ref{lem:glqmod} that good quasi-modifications $\bfY_i\to \bfX$ can be glued to a modification of $\bfX$ after blowing them up. These two results together with the quasi-compactness of $\Val(\bfX)$ imply Nagata's compactification theorem for algebraic spaces, see Theorem~\ref{th:decompose1fintype}. Finally, we apply the approximation theory of Rydh \cite{rydapr} to remove unnecessary finite type assumptions, and complete the proof of Theorem~\ref{thm:modmainth}.

\subsubsection{Main results on Riemann-Zariski spaces}
The quasi-modifications constru\-cted in Theorem~\ref{strictqmodth} are {\em strict}, i.e., satisfy $\calY=\calX$. In Section~\ref{rzsec} we study the whole family of quasi-modifications along an adic semivaluation, and prove that there exist ``enough" (not necessarily strict) quasi-modifications, see Theorem~\ref{thm:qmodth} for precise formulation. As an easy corollary, we complete our investigation of the topological spaces $\RZ(\bfX)$, $\Val(\bfX)$, and $\Spa(\bfX)$. Namely, we show in Theorem~\ref{topvalth} that any open quasi-compact subset in $\Val(\bfX)$ is the image of $\Val(\bfW)$ for a quasi-modification $\bfW\to\bfX$, and conclude that the map $\red_\bfX$ is a homeomorphism (Theorem~\ref{rzth}), and $r_\bfX$ is a topological quotient (Corollary~\ref{rzcor}). In particular, all three topologies on $\Val(\bfX)$ coincide: the topologies generated by (1) pseudo-modifications, (2) quasi-modifications, and  (3) compactifiable quasi-modifications.

We conclude the paper with two other applications of our theory. Theorem~\ref{schemeth} establishes a valuative criterion for schematization of algebraic spaces, and Theorem~\ref{prufmorth} describes {\em Pr\"ufer morphisms} of algebraic spaces, i.e., morphisms that admit no non-trivial modifications.


\section{Models}\label{modelsec}

\subsection{Definition and basic properties}

\subsubsection{The category of models}
A {\em model} $\bfX$ is a triple $(\calX, X, \psi_\bfX)$ such that $X$ and $\calX$ are quasi-compact and quasi-separated (qcqs) algebraic spaces, and $\psi_\bfX\colon\calX\to X$ is a separated schematically dominant morphism. The spaces $X$ and $\calX$ are called the {\em model space} and the {\em generic space} of $\bfX$ respectively. We always denote the model, the model space, and the generic space by the same letter, but for model we use bold font, for the generic space calligraphic font, and for the model space the regular font. If both $X$ and $\calX$ are schemes then the model $\bfX$ is called {\em schematic}.

A {\em morphism of models} $\bff\colon\bfY\to \bfX$ is a pair of morphisms $\gtf\colon \calY\to \calX$ and $f\colon Y\to X$ such that $\psi_\bfX\circ \gtf=f\circ \psi_\bfY$. The maps $\gtf$ and $f$ are called the {\em generic} and the {\em model components} of $\bff$. Once again, we use the same letter but different fonts to denote the morphisms of models and their components. We say that a morphism $\bff\colon\bfY\to \bfX$ is {\em separated} if $f\colon Y\to X$ is separated. In this case,
$\psi_\bfX\circ \gtf=f\circ \psi_\bfY$ is separated, and hence $\gtf\:\calY\to\calX$ is also separated.

\subsubsection{Good models}\label{goodsec}
As we have mentioned in the introduction, the class of models $\bfX$ for which $\psi_\bfX$ is an open immersion plays an important role in the sequel. We call such models {\em good}.

\begin{rem}\label{goodrem}
(i) Any good model $\bfX$ can be equivalently described by its model space $X$ and the closed subset $Z=|X|\setminus|\calX|$. Since $\calX$ is quasi-compact, $Z=|\bfSpec(\calO_X/\calI)|$ for a finitely generated ideal that we call an {\em ideal of definition} of $\bfX$. Any other ideal of definition $\calJ$ satisfies $\calI^m\subseteq\calJ^n\subseteq\calI$ for $m\gg n\gg 0$.

(ii) We noted in the introduction that there is a large similarity between good models $\calX\into X$ and formal models $\gtX$ of non-archimedean spaces $\gtX_\eta$. In particular, $X$ corresponds to $\gtX$, $\calX$ corresponds to the generic fiber of $\gtX$, $Z$ corresponds to the closed fiber $\gtX_s$, and ideals of definition of $\bfX$ correspond to ideals of definition of $\gtX$.
\end{rem}

\subsubsection{Basic properties of models}
The proof of the following result is almost obvious, so we omit it.

\begin{prop}\label{prop:basicpropofmodels}
(1) The functor $X\mapsto (X,X,id_X)$ is a full embedding of the category of qcqs algebraic spaces into the category of models. By abuse of notation we shall not distinguish between $X$ and $(X,X,id_X)$.

(2) Any model $\bfX$ admits two canonical morphisms $\calX\to \bfX\to X$.

(3) The category of models admits fiber products. Furthermore, if $\bfY\to\bfX$ and $\bfW\to\bfX$ are morphisms of models then the generic space of $\bfY\times_\bfX\bfW$ is $\calY\times_\calX\calW$, the model space of $\bfY\times_\bfX\bfW$ is the schematic image of $\calY\times_\calX\calW\to Y\times_XW$, and $\psi_{\bfY\times_\bfX\bfW}$ is the natural map.
\end{prop}

\subsubsection{Classes of models}
Our next aim is to introduce basic classes of models and their morphisms that will be used in the sequel.

\begin{defin}
(1) Let $(\star)$ be one of the following properties of morphisms: affine, of finite type, of finite presentation. We say that a model $\bfX$ satisfies $(\star)$ if and only if $\psi_\bfX$ satisfies $(\star)$.

(2) A model $\bfX$ is called {\em affinoid} if both $X$ and $\calX$ are affine.

(3) The {\em diagonal} of a morphism $\bff \colon\bfY\to\bfX$ is the natural morphism
$$\Delta_\bff\:\calY\to \calX\times_XY.$$
\end{defin}

\subsubsection{Adic morphisms}
A morphism $\bff\:\bfY\to\bfX$ is called {\em adic} (resp. {\em semi-cartesian}) if its diagonal is proper (resp. a clopen immersion).

\begin{rem}
To illustrate importance of adic morphisms and to explain the terminology, let us consider the case of a morphism $\bff\:\bfY\to\bfX$ of good models. Set $T=|Y|\setminus|\calY|$ and $Z=|X|\setminus|\calX|$. Then it is easy to see that $\bff$ is adic if and only if $f^{-1}(Z)=T$, or, equivalently, the pullback of an ideal of definition is an ideal of definition. Using the analogy with formal schemes from Remark~\ref{goodrem} we see that adic morphisms of good models correspond to adic morphisms of formal schemes.
\end{rem}

\begin{lem}\label{propcomplem}
(1) Let $\bff\:\bfY\to\bfX$ be a morphism of models. If $\gtf$ is proper and $f$ is separated then $\bff$ is adic.

(2) Composition of adic morphisms is adic.

(3) Adic morphisms are preserved by base changes.

(4) Let $\bff\:\bfY\to\bfX$ and $\bfg\:\bfZ\to\bfY$ be morphisms of models. If $\bff\circ\bfg$ is adic and $\gtf$ is separated then $\bfg$ is adic.
\end{lem}
\begin{proof}
(1) The base change $g\:\calX\times_X Y\to\calX$ of $f$ is separated and the composition $\gtf=g\circ\Delta_\bff$ is proper, hence $\Delta_\bff$ is proper.

(2) Let $\bff\:\bfY\to\bfX$ and $\bfg\:\bfZ\to\bfY$ be two adic morphisms. Then $\Delta_\bff$ and $\Delta_\bfg$ are proper. Hence the base change $h\:\calY\times_YZ\to \calX\times_XZ$ of $\Delta_\bff$ is proper, and the composition $\Delta_{\bfg\circ\bff}=h\circ\Delta_{\bfg}$ is proper too. Thus, $\bfg\circ\bff$ is adic.

(4) The composition $\calZ\to Z\times_Y\calY\to Z\times_X\calX$ is proper, and $Z\times_Y\calY\to Z\times_X\calX$ is separated since so are $\gtf$ and the diagonal $Y\to Y\times_XY$. Thus, $\calZ\to Z\times_Y\calY$ is proper.
We leave the proof of (3) to the reader.
\end{proof}

\subsubsection{Immersions}
An adic morphism $\bff\:\bfY\to\bfX$ is called an {\em immersion} (resp. {\em open immersion}, resp. {\em closed immersion}) if both its components are so. In particular, by an {\em open submodel} of $\bfX$ we mean a model $\bfU$ together with an open immersion $\bfi\:\bfU\to\bfX$. If $\{\bfU_\alp\subseteq \bfX\}_{\alpha\in A}$, are open submodels then the {\em union} $\cup_{\alp\in A}\bfU_\alp$ and the {\em intersection} $\cap_{\alp\in A}\bfU_\alp$ are defined componentwise. We say that $\{\bfU_\alp\}_{\alpha\in A}$ is an {\em open covering} of $\bfX$ if $\cup_{\alp\in A}\bfU_\alp=\bfX$.

\begin{rem}\label{rem:OpenImmIsCart}
It is easy to see that any open immersion $\bff\:\bfU\to \bfX$ is semi-cartesian: Indeed, the morphism $\calX\times_XU\to\calX$ and the composition $\calU\to \calX\times_XU\to\calX$ are open immersions, hence so is the diagonal $\delta_\bff\:\calU\to \calX\times_XU$. Since $\bff$ is adic, $\delta_\bff$ is proper and hence $\bff$ is semi-cartesian. However, a closed immersion need not be semi-cartesian, for example, take $\bfY=\Spec(k)$, $\bfX=(\bbP^1_k\to\Spec(k))$, $f=\Id_{\Spec(k)}$, and $\gtf$ is the closed immersion of the origin of $\bbP^1_k$.
\end{rem}

\subsection{Modifications of models}
\begin{defin}\label{moddef}
Let $\bff\:\bfY\to\bfX$ be a morphism of models.

(1) $\bff$ is called a {\em pseudo-modification} if $\gtf$ is an open immersion and $f$ is a separated morphism of finite type. A pseudo-modification is called {\em strict} if $\gtf$ is an isomorphism.

(2) $\bff$ is called a {\em quasi-modification} if it is an adic pseudo-modification.

(3) $\bff$ is called a {\em modification} if it is a strict pseudo-modification and $f$ is proper.
\end{defin}

By abuse of language, we will often refer to $\bfY$ as a {\em modification} or {\em quasi-modification} or {\em pseudo-modification} of $\bfX$. In particular, we will say that a modification $\bfY\to \bfX$ is {\em good (resp. affine good)} if $\bfY$ is so. The following easy observation will be useful:

\begin{prop}\label{prop:strpmISqm}
Any strict pseudo-modification is a quasi-modification. In particular, any modification is a quasi-modification.
\end{prop}

\begin{exam}\label{modexam}
(i) If $\bfY\to\bfX$ is a modification then any open submodel $\bfU\subseteq\bfY$ is a quasi-modification of $\bfX$. In Theorem~\ref{thm:modmainth} (a strong version of Nagata's compactification), we will prove that a partial converse is true. Namely, for any quasi-modification $\bfU\to\bfX$ there exist modifications $\bfX'\to\bfX$ and $\bfU'\to\bfU$ such that $\bfU'\subseteq\bfX'$ is an open submodel, thereby justifying the terminology.

(ii) A simple example of a pseudo-modification, which is not a quasi-modification is the following: $\bfX=\bbA^1$, $\bfU=(\bbG_m\into\bbA^1)$, and $\bfU\to\bfX$ is the natural map.
\end{exam}

\begin{rem}
Although, the notion of quasi-modification seems to be new, it is tightly related to ideas used in various proofs of Nagata's compactification theorem for schemes. Assume that $\bfX$ is a schematic model.

(i) If $\bfX$ is a good model then quasi-modifications of $\bfX$ are closely related to {\em quasi-dominations} that were introduced by Nagata and played an important role in the proofs of Nagata's compactification theorem by Nagata, Deligne, and Conrad. See, for example \cite[\S2]{con}.

(ii) Modifications appeared in \cite{temrz} under the name of $\calX$-modifica\-tions. Namely, $\bff\:\bfY\to\bfX$ is a modification if and only if $\gtf$ is an isomorphism and $Y$ is an $\calX$-modification of $X$. Quasi-modifications were not introduced terminologically in \cite{temrz}, but they were used in the paper (e.g., see \cite[Prop.~3.3.1]{temrz}).
\end{rem}

Next, we establish basic properties of pseudo-modifications.

\begin{lem}
Let $\bff\:\bfY\to\bfX$ be a morphism of models. If the generic component $\gtf$ is an immersion then so is the diagonal $\Delta_\bff$. In particular, if $\bff$ is a pseudo-modification then $\Delta_\bff$ is an immersion.
\end{lem}
\begin{proof}
Obvious.
\end{proof}

\begin{lem}\label{modlem}
(1) Modifications, quasi-modifications, and pseudo-modifications, are preserved by compositions and base changes. In particular, if $\bfY\to\bfX$ and $\bfZ\to\bfX$ are modifications (resp. quasi-modifications, resp. pseudo-modifications) of models then so is $\bfY\times_\bfX\bfZ\to\bfX$.

(2) For any model $\bfX$, the families of modifications, quasi-modifications, and pseudo-modifications of $\bfX$ are filtered.
\end{lem}
\begin{proof}
The second claim follows from the first one. The latter is obvious for modifications and pseudo-modifications, and the case of quasi-modifications follows from Lemma \ref{propcomplem}.
\end{proof}

For good models, one can say more.

\begin{lem}\label{normmodlem}
Let $\bff\:\bfY\to\bfX$ be a pseudo-modification. If $\bfX$ is good then,

(1) $\bfY$ is good and the diagonal $\Delta_\bff$ is an open immersion.

(2) $\bff$ is a quasi-modification if and only if it is semi-cartesian.
\end{lem}
\begin{proof}
(1) The open immersion $\calY\into\calX\into X$ factors through $\psi_\bfY\:\calY\to Y$, hence $\psi_\bfY$ is an immersion by \cite[Tag:07RK]{stacks}. Since $\psi_\bfY$ is schematically dominant, it is, in fact, an open immersion.

(2) By definition, a pseudomodification $\bff$ is quasi-modification if and only if the open immersion $\Delta_\bff$ is proper if and only if $\bff$ is semi-cartesian.
\end{proof}

\subsubsection{Riemann-Zariski spaces of models}
By the {\em Riemann-Zariski} or {\em RZ space} of a model $\bfX$ we mean the topological space $\RZ(\bfX)=\underleftarrow{\lim}_{\alp\in A}|X_\alp|$, where $\bff_\alp{\colon}\bfX_\alp\to \bfX$, $\alp\in A$, is the family of all modifications of $\bfX$. We claim that $\RZ$ is a functor. Indeed, if $\bff\:\bfY\to \bfX$ is a morphism of models then for any modification $\bfZ\to \bfX$ the base change $\bfZ\times_\bfX\bfY\to\bfY$ is a modification by Lemma~\ref{modlem}, and therefore there exists a natural continuous map $\RZ(\bff)\:\RZ(\bfY)\to\RZ(\bfX)$.

\begin{rem}\label{rem:RZextrastucture}
One may wonder whether the limit of $X_\alp$'s exists in a finer category. It is easy to see that the limit does not have to be representable in the category of algebraic spaces, so it is natural to seek for an enlargement of this category. For example, in the case of schematic models, a meaningful limit exists in the category of locally ringed spaces (see \cite{temrz}). It seems that in our situation the most natural framework is the category of strictly henselian topoi as introduced by Lurie (intuitively, this is a ringed topos such that the stalks of the structure sheaf are strictly henselian rings). However, even the foundations of such theory are only being elaborated; for example, the fact that algebraic spaces embed fully faithfully into the bicategory of such topoi was checked very recently by B. Conrad, see \cite[Th. 3.1.3]{conrad}). For this reason, we prefer not to develop this direction in the paper.
\end{rem}

\subsection{Blow ups}
\begin{defin}
Let $\bfX$ be a model, $\calJ\subset \calO_X$ an ideal of finite type such that $\psi_\bfX^{-1}\calJ=\calO_\calX$, and $\psi\:\calX\to Bl_\calJ(X)$ the natural map. Then the model $Bl_\calJ(\bfX):=(\calX, Bl_\calJ(X), \psi)$ is called the {\em blow up of $\bfX$ along $\calJ$}. An open submodel of a blow up is called a {\em quasi-blow up}.
\end{defin}

Any blow up (resp. quasi-blow up) is a modification (resp. quasi-modification). We will only consider blow ups of good models which can be identified with usual $\calX$-admissible blow ups of $X$ by the following remark.

\begin{rem}
Let $\bff\:\bfY\to\bfX$ be a blow up. If $\bfX$ is good then $\calY\toisom\calX$, $Y\to X$ is an $\calX$-admissible blow up, and $\bfY$ is good. Conversely, for an $\calX$-admissible blow up $Y\to X$, the model $\bfY:=(\calX\to Y)$ is good and $\bfY\to\bfX$ is a blow up of models.
\end{rem}

\subsubsection{Extension of blow ups}
An important advantage of blow ups over arbitrary modifications is that they can be easily extended from open submodels. This property is very useful in Raynaud's theory, and it will be crucial in \S\ref{globalsec} while gluing good quasi-modifications.

\begin{prop}\label{prop:extofblups}
Let $\bfX$ be a good model, $\bfU\into\bfX$ an open submodel, and $\bff\:\bfU'\to \bfU$ a blow up. Then there exists a blow up $\bfX'\to\bfX$ such that $\bfU'=\bfX'\times_\bfX\bfU$.
\end{prop}
\begin{proof}
The spaces $U$, $\calX$, and $\calU$ are open subspaces of $X$, and $\calU=U\cap\calX$ by Remark~\ref{rem:OpenImmIsCart}. By definition, $\bfU'=Bl_{\calI}(\bfU)$ for a finite type ideal $\calI\subseteq\calO_U$ such that the subspace $W\subset U$ defined by $\calI$ is disjoint from $\calU$. Then $W$ is also disjoint from $\calX$, and hence so is the schematic closure of $W$ in $X$, that we denote by $Z$. By \cite[Theorem A]{rydapr}, the ideal $I_Z\subset\calO_X$ defining $Z$ is the filtered colimit of ideals of finite type $\calJ_\alpha$ that extend $\calI$ to $X$. By quasi-compactness of $\calX$, there exists $\alpha$ such that $\calX$ is disjoint from the subspace defined by $\calJ_\alpha$. Then $\bfX':=Bl_{\calJ_\alpha}(\bfX)$ is as needed.
\end{proof}

\begin{prop}\label{prop:compqblups}
Let $\bfX$ be a good model. If $\bfZ\to\bfY$ and $\bfY\to\bfX$ are blow ups (resp. quasi-blow ups) then so is the composition $\bfZ\to\bfX$.
\end{prop}
\begin{proof}
For schematic models the case of blow ups follows from the fact that composition of $\calX$-admissible blow ups of the algebraic space $X$ is an $\calX$-admissible blow up by \cite[Lem. 5.1.4]{RG}. Note that the proof in \cite{RG} is incomplete, and the complete proof (including an additional argument due to Raynaud) can be found in \cite[Lemma 1.2]{con}. In general, the proof is almost identical to the proof in \cite{con}, and the only difference is that one must replace Zariski-local arguments with \'etale-local, and use the approximation theory of Rydh \cite[Corollary 4.12]{rydapr} instead of \cite[I, 9.4.7; $\rm{IV_1}$, 1.7.7]{ega}.

Assume, now, that the morphisms are quasi-blow ups. Then $\bfY\subseteq\bfY'$ and $\bfZ\subseteq\bfZ'$ for blow ups $\bfY'\to\bfX$ and $\bfZ'\to\bfY$. The latter blow up can be extended to a blow up $\bfZ''\to\bfY'$ by Proposition~\ref{prop:extofblups}. So, the composition $\bfZ''\to\bfY'\to\bfX$ is a blow up, and we obtain that $\bfZ\subset\bfZ'\subset\bfZ''$ is a quasi-blow up of $\bfX$.
\end{proof}

\section{Relative birational spaces}\label{relbirsec}

\subsection{Semivaluations of models}

\subsubsection{Valuation models}
We refer to \cite[\S5.1.1]{pruf} for the definition of valuation algebraic spaces. By a {\em valuation model} we mean a model $\bfT$ such that $T$ is a valuation algebraic space and $\calT$ is its generic point.

\subsubsection{Valuative diagrams}
A valuative diagram as defined in \cite[\S5.2.1]{pruf} can be interpreted as a morphism $\bfv\:\bfT\to\bfX$ from a valuation model $\bfT$ to a model $\bfX$ such that $v\:T\to X$ is separated.

\subsubsection{Semivaluations}\label{semivalsec}
By a {\em semivaluation} of a model $\bfX$ we mean a valuative diagram $\bfT\to\bfX$ such that $\calT\to \calX$ is a Zariski point. It follows from \cite[Lemma~5.2.3]{pruf} that any valuative diagram $\bfS\to\bfX$ factors uniquely through a semivaluation $\bfT\to\bfX$ such that $T\to S$ is surjective. As a corollary, it is shown in \cite[Proposition~5.2.9]{pruf} that semivaluations can be used to test properness of $\psi_\bfX$.

\subsubsection{The sets $\Spa(\bfX)$}
The set of all semivaluations of $\bfX$ is denoted by $\Spa(\bfX)$. For any separated morphism of models $\bff\:\bfY\to\bfX$ and a semivaluation $\bfT\to\bfY$, the composition $\bfT\to\bfX$ is a valuative diagram, hence induces a semivaluation of $\bfX$ by \cite[Lemma~5.2.3]{pruf}. This defines a map $\Spa(\bff)\:\Spa(\bfY)\to\Spa(\bfX)$ in a functorial way, and hence $\Spa$ is a functor from the category of models with separated morphisms to the category of sets.

\subsubsection{Adic semivaluations}\label{adicvalsec}
We say that a semivaluation $\bfv\:\bfT\to\bfX$ is adic if it is adic as a morphism of models. It is proved in the refined valuative criterion \cite[Theorem~5.2.14]{pruf} that adic semivaluations suffice to test properness of $\psi_\bfX$.

\begin{rem}\label{rem:spa}
(i) If $\bfX$ is schematic then semivaluations of $\bfX$ are exactly the $X$-valuations of $\calX$ as defined in \cite[\S3.1]{temrz}.

(ii) For the sake of comparison, we note that in adic geometry of R. Huber, one defines affinoid spaces to be the sets $\Spa(A,A^\rhd)$ of all continuous semivaluations of the morphism $\Spec(A)\to\Spec(A^\rhd)$, equipped with a natural topology and a structure sheaf. The continuity condition is empty if the topology on $A$ is discrete, and is equivalent to adicity of the semivaluation otherwise.
\end{rem}

\begin{lem}\label{adicsemivallem}
Assume that $\bfv\:\bfS\to\bfX$ is a semivaluation and $\bff\:\bfT\to\bfS$ is a separated morphism of valuation models such that $f$ is surjective. If the composed  valuative diagram $\bfu\:\bfT\to\bfX$ is adic then $\bfv$ is adic.
\end{lem}
\begin{proof}
By \cite[Lemma~5.2.12]{pruf}, if $\bfv$ is not adic then the morphism $\calS\to\calX$ extends to a pro-open subspace $U\into S$ strictly larger than $\calS$. Therefore, $\calT\to\calX$ extends to the preimage $V\into T$ of $U$. Since $V$ is strictly larger than $\calT$ by the surjectivity of $f$, we obtain that $\bfu$ is not adic, which is a contradiction.
\end{proof}

\subsubsection{The sets $\Val(\bfX)$}\label{subsec:adicsemival}
We denote by $\Val(\bfX)$ the set of all adic semivaluations of a model $\bfX$. It is a subset of $\Spa(\bfX)$, and we denote the embedding map by $\iota_\bfX\:\Val(\bfX)\into\Spa(\bfX)$. If $\bff\:\bfY\to\bfX$ is a separated adic morphism of models then for an adic semivaluation $\bfT\to\bfY$ the composition $\bfT\to\bfX$ is adic by Lemma~\ref{propcomplem}, and hence $\Spa(\bff)$ takes $\Val(\bfY)$ to $\Val(\bfX)$ by Lemma~\ref{adicsemivallem}. Set $\Val(\bff):=\Spa(\bff)|_{\Val(\bfY)}$ making $\Val$ to a functor from the subcategory of models with separated adic morphisms to the category of sets.

\subsubsection{The retraction $\Spa(\bfX)\to\Val(\bfX)$}\label{retrsec}
Let $\bfT\to\bfX$ be a semivaluation. By \cite[Lemma~5.2.12(i)]{pruf}, there exists a maximal pro-open subspace $W\into T$ such that the $X$-morphism $\calT\to\calX$ extends to $W$. Moreover, $W$ is quasi-compact and has a closed point $w$. Set $S$ to be the closure of $w$ in $T$, $\calS:=w$, and $\bfS:=(\calS\to S)$. Then $r_\bfX(\bfT\to\bfX):=(\bfS\to\bfX)$ is an adic semivaluation by \cite[Lemma~5.2.12(ii)]{pruf}.

\subsubsection{Full functoriality of $\Val$}\label{fullsec}
It is easy to see that if $\bff\:\bfY\to\bfX$ is an adic morphism then the retractions $r_\bfY$ and $r_\bfX$ are compatible with the maps $\Spa(\bff)$ and $\Val(\bff)$, i.e., $r_\bullet$ is a natural transformation of these functors.

Furthermore, we can use the retractions to extend the functor $\Val$ to the whole category of models with separated morphisms. Indeed, for a separated morphism $\bff\:\bfY\to\bfX$ we simply define $\Val(\bff)$ to be the composition $r_\bfX\circ\Spa(\bff)\circ \iota_{\bfY}$.

\subsubsection{The reduction maps}
By the {\em center} of a semivaluation $\bfv\:\bfT\to\bfX$ we mean the image $x\in X$ of the closed point of $T$. The maps $\opi_\bfX\:\Spa(\bfX)\to X$ and $\pi_\bfX\:\Val(\bfX)\to X$ that associate to a semivaluation its center are called the {\em reduction maps}. Plainly, $\opi_\bfX=\pi_\bfX\circ r_\bfX$, and the reduction maps $\opi_\bullet$ and $\pi_\bullet$ commute with $\Spa(\bff)$ and $\Val(\bff)$ for any separated morphism $\bff\:\bfY\to\bfX$.

\begin{prop}\label{prop:surjofthered}
The reduction maps $\opi_\bfX$ and $\pi_\bfX$ are surjective.
\end{prop}
\begin{proof}
It is sufficient to prove the surjectivity of $\opi_\bfX$ since $\opi_\bfX=\pi_\bfX\circ r_\bfX$. Without loss of generality we may assume that $\bfX$ is affinoid since $\opi_\bfX$ commutes with $\Spa(\bff)$ for separated morphisms $\bff$. Let $x\in X$ be a point. By schematic dominance of $\psi_\bfX$, there exists a point $y\in \calX$ such that $u:=\psi_\bfX(y)$ is a generalization of $x$. Indeed, for affine schemes this follows from \cite[Tag:00FK]{stacks}, and the general case follows by passing to an affine presentation. Take any valuation ring $R$ of $k(u)$ such that $u\to X$ extends to a morphism $\Spec(R)\to X$ that takes the closed point to $x$, and let $R'$ be any valuation ring of $k(y)$ that extends $R$. Set $\bfT:=(y\to\Spec(R'))$. Then the natural semivaluation $\bfv\:\bfT\to\bfX$ belongs to $\opi_\bfX^{-1}(x)$.
\end{proof}
\begin{rem}
The term ``reduction map" is grabbed from the theory of formal models of non-archimedean spaces. Another reasonable name would be a ``specialization map".
\end{rem}

\subsection{Topologies}\label{subsubsec:top}

\subsubsection{Featured subsets of $\Spa(\bfX)$ and $\Val(\bfX)$}

\begin{lem}\label{lem1}
Assume that $\bff\:\bfY\to\bfX$ is a pseudo-modification. Then,

(1) The map $\Spa(\bff)$ is injective.

(2) If $\bff$ is a quasi-modification then $\Val(\bff)$ is injective.

(3) If $\bfT\to\bfX$ is an adic semivaluation that lifts to $\bfY$ then $\bfT\to\bfY$ is also adic.

(4) The following are equivalent: (a) $\bff$ is a modification, (b) $\Spa(\bff)$ is a bijection, (c) $\Spa(\bff)$ takes $\Val(\bfY)$ to $\Val(\bfX)$ and induces a bijection between them.
\end{lem}
\begin{proof}
Assertion (1) follows from \cite[Proposition~5.2.6]{pruf} applied to the morphism $Y\to X$. Assertion (2) follows from (1) by the definition of $\Val$ (cf. \S~\ref{subsec:adicsemival}). (3) follows from Lemma~\ref{propcomplem} (4). It remains to prove (4).

Note that for any point $\eta\in\calX$ there exists a trivial semivaluation $\eta\to \bfX$, which is easily seen to be adic. Thus, if either $\Spa(\bff)$ or $\Val(\bff)$ is bijective then $\bff$ is strict. Hence we may assume that $\bff$ is strict.

If either (b) or (c) holds then $f$ is proper by the strong valuative criterion of properness \cite[Theorem~5.2.14]{pruf}, and hence (a) holds. Vice versa, if (a) holds then $\Spa(\bff)$ is injective by (1) and is surjective by \cite[Proposition~5.2.9]{pruf}, and hence (b) holds. Furthermore, $\Spa(\bff)$ takes $\Val(\bfY)$ to $\Val(\bfX)$ (cf. \S~\ref{subsec:adicsemival}), and hence (c) holds by (3).
\end{proof}

Thanks to Lemma~\ref{lem1}, for a pseudo-modification $\bff\:\bfY\to\bfX$, we will freely identify $\Spa(\bfY)$ with the subset of all semivaluations of $\bfX$ that lift to $\bfY$. If $\bff$ is a quasi-modification then we will use similar identification for the sets $\Val$.

\begin{defin}
Let $\bfv\:\bfT\to \bfX$ be a semivaluation. We say that a pseudo-modification $\bfY\to \bfX$ is {\em along} $\bfv$ (or, by abuse of language, along $\bfT$) if $\bfv\in\Spa(\bfY)$.
\end{defin}

\begin{lem}\label{alonglem} Let $\bfv\:\bfT\to\bfX$ be a semivaluation. Then the families of all pseudo-modifications and quasi-modifications of $\bfX$ along $\bfv$ are filtered.
\end{lem}
\begin{proof}
Follows from Lemma~\ref{modlem}.
\end{proof}

\begin{lem}\label{qmodrem}
Assume $\bfX$ is a model with a semivaluation $\bfv$, and let $\bfu:=r_\bfX(\bfv)$ be the associated adic semivaluation. Then,

(1) Any pseudo-modification $\bfg\:\bfZ\to\bfX$ along $\bfu$ is also along $\bfv$.

(2) A quasi-modification $\bff\:\bfY\to\bfX$ is along $\bfu$ if and only if it is along $\bfv$.
\end{lem}
\begin{proof}
Let $\bfT\to\bfX$ and $\bfS\to\bfX$ denote the morphisms of $\bfv$ and $\bfu$, respectively. Recall, that $\calS\in X$ is the closed point of the maximal pro-open subspace $W$ for which the $X$-morphism $\calT\to \calX$ extends to $W$, and $S$ is the closure of $\calS$ in $T$. Since $\calZ\subseteq \calX$ is open and $\calS\in\calZ$, the morphism $W\to\calX$ factors through $\calZ$. By \cite[Proposition~4.3.11]{pruf}, $T=S\coprod_{\calS}W$ is the Ferrand pushout. Thus, the compatible morphisms $W\to \calY\to Z$ and $S\to Z$ define a morphism $T\to Z$, and hence $\bfv\in\Spa(\bfZ)$. This proves (1).

To prove (2) we assume that $\bff$ is along $\bfv$. Identify $\Spa(\bfY)$ with a subset of $\Spa(\bfX)$. Since $\bff$ is adic, we have that $\bfu=\bfr_\bfX(\bfv)=\bfr_\bfY(\bfv)$, so $\bfu\in\Spa(\bfY)$.
\end{proof}

\subsubsection{Maps to $\RZ(\bfX)$}\label{mapstorzsec}
By Lemma~\ref{lem1} (4), for any modification $\bfX_\alp\to\bfX$ we can identify $\Spa(\bfX_\alp)$ with $\Spa(\bfX)$. Therefore, the reduction maps $\opi_{\bfX_\alp}$ induce a map $\Red_\bfX\:\Spa(\bfX)\to\RZ(\bfX)$, and in the same manner one obtains a map $\red_\bfX\:\Val(\bfX)\to\RZ(\bfX)$. Since the center maps are compatible with the retraction $r_\bfX$, we obtain the following diagram, in which both triangles are commutative.
\begin{equation}\label{diag:red}
\xymatrix{
\Spa(\bfX)\ar@<1ex>[d]^{r_\bfX}\ar[dr]^{\Red_\bfX} & \\
\Val(\bfX)\ar@{^{(}->}@<1ex>[u]^{\iota_\bfX}\ar[r]^{\red_\bfX} & \RZ(\bfX)\\
}
\end{equation}

Plainly, such diagrams are functorial with respect to separated morphisms of models.

\subsubsection{Topologies on $\Spa(\bfX)$}
If $\bfY$ and $\bfZ$ are two pseudo-modifications of $\bfX$ then $\Spa(\bfY)\cap\Spa(\bfZ)=\Spa(\bfY\times_\bfX\bfZ)$. Therefore, the collection of sets $\Spa(\bfY)$ for all pseudo-modifications $\bfY\to\bfX$, forms a base of a topology on $\Spa(\bfX)$, which we call {\em Zariski topology}. Since pseudo-modifications are preserved by base changes, the maps $\Spa(\bff)$ are continuous. This, enriches $\Spa$ to a functor whose target is the category of topological spaces.

In addition, we consider the boolean algebra generated by the sets $\Spa(\bfY)$ for pseudo-modifications $\bfY\to\bfX$, and call its elements {\em constructible} subsets of $\Spa(\bfX)$. They form a base of a topology, which we call {\em constructible topology}. Clearly, the maps $\Spa(\bff)$ are continuous with respect to the constructible topologies as well.

\begin{rem}
We will show in Proposition~\ref{compprop} that the sets $\Spa(\bfY)$ are quasi-compact, and so our ad hoc definition of the constructible topology coincides with what one usually takes for the definition, i.e., the topology associated to the boolean algebra generated by open quasi-compact sets. However, we will use both topologies in the proof, so it is convenient to start with the ad hoc definition.
\end{rem}


\subsubsection{The topology on $\Val(\bfX)$}
We define {\em Zariski topology} on $\Val(\bfX)$ to be the induced topology from the Zariski topology on $\Spa(\bfX)$, i.e., the sets $\Spa(\bfY)\cap\Val(\bfX)$ for pseudo-modifications $\bfY\to\bfX$ form a base of the Zariski topology. Note that by Lemma~\ref{lem1} (3), $\Spa(\bfY)\cap\Val(\bfX)=\Val(\bfY)\cap\Val(\bfX)$.

\begin{rem}
(i) The functoriality of (the enriched) $\Val$ with respect to adic morphisms is obvious, while the question about general functoriality is much more subtle. In fact, our main results about the topology of $\Val$ are the following: (1) $\red_\bfX\:\Val(\bfX)\to\RZ(\bfX)$ is a homeomorphism (Theorem~\ref{rzth}), and (2) the retraction map $r_\bfX$ is open and continuous (Corollary~\ref{rzcor}). Since the construction of $\RZ(\bfX)$ is functorial in a natural way, we will obtain an interpretation of the functoriality of $\Val$ that does not involve the retractions $r_\bfX$. Moreover, this will extend the functoriality of $\Val$ to all morphisms, and will imply the continuity of the maps $\Val(\bff)$ for an arbitrary morphism $\bff\:\bfY\to\bfX$.

(ii) One could define {\em quasi-modification topology} on $\Val(\bfX)$ by using only quasi-modifications in the definition. It is {\em a priori} weaker than the Zariski topology, but {\em a posteriori} the two topologies coincide (Theorem~\ref{topvalth}).

(iii) As one might expect, the quasi-modification topology is not so natural on $\Spa(\bfX)$. In fact, it does not distinguish points in the fibers of $\bfr_\bfX$ by Lemma~\ref{qmodrem}, while the pseudo-modification topology does so by Theorem~\ref{thm:ValT0}.
\end{rem}

\subsection{Topological properties}
In this section we establish relatively simple topological properties, whose proof does not involve a deep study of quasi-modifications.

\begin{prop}\label{prop:cont}
The maps $\iota_\bfX$, $\pi_\bfX$, $\opi_\bfX$, $\red_\bfX$, and $\Red_\bfX$ are continuous in Zariski topology.
\end{prop}
\begin{proof}
The continuity of $\iota_\bfX$ follows immediately from the definitions. For the continuity of $\pi_\bfX$ and $\opi_\bfX$, let $U\subset X$ be an open subset. Set $\bfU:=\bfX\times_XU$. Then $\opi_\bfX^{-1}(U)=\Spa(\bfU)$ and $\pi_\bfX^{-1}(U)=\Val(\bfU)$ are open subsets. Finally, since $\RZ(\bfX)$ is the limit of $|X_\alp|$, where $\bff_\alp{\colon}\bfX_\alp\to \bfX$, $\alp\in A$, is the family of all modifications of $\bfX$, the continuity of $\red_\bfX$ and $\Red_\bfX$ follows from the continuity of $\pi_{\bfX_\alp}$ and $\opi_{\bfX_\alp}$.
\end{proof}

We conclude Section~\ref{relbirsec} by showing that $\Val(\bfX)$ and $\Spa(\bfX)$ are quasi-compact. For the proof we need the following observation.

\begin{lem}\label{lem1.5}
For any $\bfv\in \Spa(\bfX)$, the retraction $r_\bfX(\bfv)\in\Val(\bfX)\subseteq\Spa(\bfX)$ is a specialization of $\bfv$ in Zariski topology.
\end{lem}
\begin{proof}
Follows from Lemma~\ref{qmodrem}(1).
\end{proof}

\begin{prop}\label{compprop}
Let $\bfX$ be a model, and $Z\subseteq |X|$ a pro-constructible subset. Then the sets $\opi_\bfX^{-1}(Z)\subseteq \Spa(\bfX)$ and $\pi_\bfX^{-1}(Z)\subseteq \Val(\bfX)$ are quasi-compact in the Zariski topology. In particular, $\Spa(\bfX)$ and $\Val(\bfX)$ are quasi-compact. Furthermore, $\opi_\bfX^{-1}(Z)\subseteq \Spa(\bfX)$ is compact in the constructible topology.
\end{prop}
\begin{proof}
First, we claim that if $\opi_\bfX^{-1}(Z)\subseteq \Spa(\bfX)$ is quasi-compact in Zariski topology then so is $\pi_\bfX^{-1}(Z)\subseteq \Val(\bfX)$. Indeed, any open covering $\gtW$ of $\pi_\bfX^{-1}(Z)$ in $\Val(\bfX)$ is the restriction of an open covering $\gtU$ of $\pi_\bfX^{-1}(Z)$ in $\Spa(\bfX)$. By Lemma~\ref{lem1.5}, for any $\bfv\in \opi_\bfX^{-1}(Z)$ the retraction $r_\bfX(\bfv)$ is a specialization of $\bfv$. Note that $r_\bfX(\bfv)\in\pi_\bfX^{-1}(Z)$ since $\opi_\bfX=\pi_\bfX\circ r_\bfX$. Thus, $\gtU$ is also an open covering of $\opi_\bfX^{-1}(Z)$ in $\Spa(\bfX)$. If $\opi_\bfX^{-1}(Z)$ is quasi-compact then its covering $\gtU$ admits a finite refinement, and hence $\gtW$ admits a finite refinement too.

Assume, first, that $\bfX$ is affinoid. This case is essentially due to Huber (see \cite{Hub1} and \cite{temrz} for details). If $X=\Spec(A)$ and $\calX=\Spec(B)$ then, as explained in \cite[\S3.1]{temrz}, the topological space $\Spa(\bfX)$ is nothing but Huber's adic space $\Spa(B,A)$, where $A$ and $B$ are viewed as discrete topological rings. In \cite{Hub1}, it is proved that $\Spa(B,A)$ is compact in the constructible topology. Since $\opi_\bfX$ is continuous, the set $\opi^{-1}_\bfX(Z)$ is pro-constructible. Thus, it is an intersection of constructible sets, which are closed, and hence compact in the constructible topology. We conclude that $\opi^{-1}_\bfX(Z)$ is compact.

For an arbitrary $\bfX$, pick an \'etale affine covering $f\:V\to X$, and set $\bfV:=\bfX\times_XV$. Now, pick an \'etale affine covering
$\calU\to\calV$, and set $U:=V$ and $W:=f^{-1}(Z)$. We obtain an affinoid model $\bfU$ and morphisms $\bfU\to\bfV\to\bfX$, whose components are \'etale and surjective. Consider the induced map $\phi\:\Spa(\bfU)\to\Spa(\bfX)$. Note that $\phi^{-1}(\opi_\bfX^{-1}(Z))=\opi_\bfU^{-1}(W)$ is compact by the already established affinoid case, hence we should only prove that $\phi$ is surjective. 

Let $\bfv\:\bfT\to\bfX$ be a semivaluation. Set $\bfS:=\bfT\times_\bfX\bfV=\bfT\times_XV$, and note that $S$ is an SLP space because it is \'etale over $T$. Applying \cite[Proposition~5.1.9(i)]{pruf} to $S\to V$ we obtain that $S$ is a scheme. Let $s\in S$ be a preimage of the closed point of $T$ and $S_s=\Spec(\calO_{S,s})$. Set $\bfS_s:=\bfS\times_SS_s$. Then $\bfw\:\bfS_s\to\bfV$ is a semivaluation, whose image in $\Spa(\bfX)$ is $\bfv$. Indeed, $\Spa(\bfv)\to\bfT$ is a surjective separated morphism of valuative spaces inducing an isomorphism of the generic points. Hence, by \cite[Lemma~5.2.3(ii)]{pruf}, it is an isomorphism. Let now $\eta$ be the generic point of $S_s$. We identify it with a point of $\calV$, and choose any preimage $\veps\in\calU$. Choose any valuation ring $A$ of $k(\veps)$ whose intersection with $k(\eta)$ coincides with $\calO_{S,s}$. Then $\veps\to\Spec(A)$ defines a valuation model and the morphisms $\veps\to\calU$ and $\Spec(A)\to S_s\to U$ define a semivaluation of $\bfU$ that lifts $\bfw$.
\end{proof}

\begin{cor}\label{cor:surjofred}
The maps $\Red_\bfX$ and $\red_\bfX$ are surjective.
\end{cor}
\begin{proof}
It is sufficient to show that $\Red_\bfX$ is surjective. Pick a point $x\in\RZ(\bfX)$, and consider the corresponding compatible family of points $x_\alp\in|X_\alp|$, where $\bfX_\alp$ are the modifications of $\bfX$. The subsets $S_\alp:=\opi_{\bfX_\alp}^{-1}(x_\alp)\subseteq\Spa(\bfX_\alp)=\Spa(\bfX)$ are compact in the constructible topology since $\opi_{\bfX_\alp}$ are continuous maps between compact spaces and $x_\alp\in |X_\alp|$ are compact. Furthermore, $S_\alp\ne\emptyset$ for each $\alp$ by Proposition~\ref{prop:surjofthered}. Finally, since the family of modifications of $\bfX$ is filtered by Lemma~\ref{modlem} (2), for any finite set of indices $\alp_1,\dotsc,\alp_n$, there exists a modification $\bfX_\alp\to\bfX$ that factors through all $\bfX_{\alp_i}$. Thus, $S_\alp\subseteq\cap_{i=1}^nS_{\alp_i}$. Hence $\cap_{i=1}^nS_{\alp_i}\ne\emptyset$, and by compactness, $Red^{-1}_\bfX(x)=\cap_\alp S_\alp\ne\emptyset$.
\end{proof}

\begin{theor}\label{thm:ValT0}
For any model $\bfX$ the spaces $\Spa(\bfX)$ and $\Val(\bfX)$ are $T_0$.
\end{theor}
\begin{proof}
Since the topology on $\Val(\bfX)$ is induced from $\Spa(\bfX)$, it suffices to prove the assertion for $\Spa(\bfX)$.

Let $\bfT,\bfS\in\Spa(\bfX)$ be two distinct points. We should construct a pseudo-modification $\bfY\to \bfX$ such that $\bfT$ belongs to the open set $\Spa(\bfY)$, but $\bfS$ does not, or vice-versa. If $\calT\ne\calS$ as points of $\calX$ then without loss of generality we may assume that $\calT$ does not belong to the closure of $\calS\in \calX$. Thus, there exists an open $\calY\subset\calX$ containing $\calT$ and not containing $\calS$. Set $Y$ to be the schematic image of $\calY\to X$, and $\bfY:=(\calY\to Y)$. Then $\bfY\to \bfX$ is a pseudo-modification, $\bfT$ belongs to $\Spa(\bfY)$, but $\bfS$ does not.

Assume now that $\calT=\calS$. Without loss of generality we may assume that $S\to X$ does not factor through $T$, and it suffices to construct a pseudo-modification $\bfY\to\bfX$ along $\bfT$ such that $S\to X$ does not factor through $Y$. This will be done in few steps as shown in the diagram
$$
\xymatrix{
\calT\ar[d]_{\psi_\bfT}\ar[r]\ar@{}[rd]|{\coprod^{X^\aff}} & Z\ar[d]^\psi\ar@{^{(}->}[r]\ar@{}[rd]|{\coprod} & \calY\ar[d]\ar@{=}[r] & \calY\ar@{^{(}->}[r]\ar[d]^{\psi_{\bfY}} & \calX\ar[d]^{\psi_\bfX}\\
T\ar[r] & U\ar[r] & Y'\ar[r] & Y\ar[r] & X}
$$
where $Z\into\calY$ is a closed immersion and $\calY\into\calX$ is an open immersion.

Pick any open affine $Z$ in the closure of $\calT$ in $\calX$. We may assume that $Z$ is $X$-affine: Indeed, let $\tilX\to X$ be an \'etale affine presentation. Set $\tilZ:=Z\times_X\tilX$. Let $W\subset\tilZ$ be an open affine subset. Then there exists $h\in\calO(Z)$ such that $Z_h\times_X\tilX\to \tilZ$ factors through $W$. Thus, after replacing $Z$ with $Z_h$ we may assume that $\tilZ$ and $\tilZ\to\tilX$ are affine, and hence so is $Z\to X$ by descent.

Let $U:=T\coprod_\calT^{X^\aff}Z$ be the $X$-affine pushout, i.e., the pushout in the category of $X$-affine algebraic spaces. Note that $U=\bfSpec_X(\calO_T\times_{\calO_\calT}\calO_Z)$. Then $Z\to U$ is separated and schematically dominant. Furthermore, $\calT=\underleftarrow{\lim}_hZ_h$ for the family of principal localizations of $Z$. We claim that $T=\underleftarrow{\lim}_hU_h$, where $U_h:=T\coprod_\calT^{X^\aff}Z_h$. To check this we can replace $X$ with an affine \'etale covering and all the $X$-spaces with their base changes, so assume that $X$ is affine. Then $X$-affine pushouts are nothing but affine pushouts, and by the very definition, $\calO(U_h)=\calO(T)\cap\calO(Z_h)$ as a subring of $\calO(\calT)=Frac(\calO(T))=Frac(\calO(Z))$. The equality $T=\underleftarrow{\lim}_hU_h$ follows, and as a consequence we obtain that the morphism $S\to X$ does not factor through $U_h$ for $h$ large enough. After replacing $Z\to U$ with $Z_h\to U_h$, we may assume that $S$ does not factor through $U$.

Let $\calY\subset\calX$ be an open subspace such that $Z\to \calY$ is a closed immersion. Let $Y':=U\coprod_Z\calY$ be the Ferrand pushout, see \cite[\S3.1]{push}. Then $U\to Y'$ is a closed immersion by \cite[Theorem~4.4.2(iii)]{push}. Since the morphism $S\to Y'$ takes the generic point $\calS=\calT$ of $S$ to $U$ and $S\to X$ does not factor through $U$, it does not factor through $Y'$ too. Finally, by approximation \cite[Lemma~2.1.10]{pruf}, we can choose $Y\to X$ separated of finite type such that $\calY\to X$ factors through $Y$, $\calY\to Y$ is affine and schematically dominant, and $S\to X$ does not factor through $Y$. Thus $\bfY:=(\calY\to Y)$ is as needed.
\end{proof}
We shall mention that a similar, but much more involved, argument will be used to prove Theorem~\ref{strictqmodth} below.

\section{Main results on modifications of models}\label{mmmsec}

\subsection{Approximating adic semivaluations with quasi-modifications}\label{strictqmodsec}
The following theorem is the key result towards Nagata's theorem. Its proof occupies whole~\S\ref{strictqmodsec}.

\begin{theor}\label{strictqmodth}
Let $\bfX$ be a finite type model, and $\bff\:\bfT\to\bfX$ an adic semivaluation. Then there exists a strict quasi-modification $\bfY\to \bfX$ along $\bfT$ such that $\bfY$ is good and affine.
\end{theor}

\subsubsection{Plan of the proof}\label{plansec}
The proof is relatively heavy, so we first provide a general outline. To ease the notation we identify $\calT$ with the corresponding point $\eta\in\calX$. Let $Z$ be the closure of $\eta$.

Step 1. Choose a suitable open $X$-affine subscheme $Z_1\into Z$ and set $U_1:=T\coprod_\eta^{X^\aff}Z_1$ to be the $X$-affine pushout of $T$ and $Z_1$. Ensure that $(Z_1\to U_1)$ is a good affine model satisfying certain conditions; the main one being adicity over $\bfX$.

Step 2. Define $U$ to be the gluing of $U_1$ and $Z$ along $Z_1$ in the Zariski topology, and show that $(Z\to U)$ is a good affine model separated over $\bfX$.

Step 3. Set $V:=U\coprod_Z\calX$ to be the Ferrand pushout (see \cite[\S3.1]{push}) and show that $(\calX\to V)$ is a good affine model separated over $\bfX$.

Step 4. Construct an affine strict quasi-modification $\bfY\to\bfX$ along $\bfT$ by choosing $Y\to X$ to be a finite type approximation of $V\to X$.

Step 5. Modify $\bfY$ so that it becomes good.


\begin{equation}\label{eq:push}
\xymatrix{
\calT\ar@{^{(}->}[d]_{\psi_\bfT}\ar[r]\ar@{}[rd]|{\coprod^{X^\aff}} & Z_1\ar@{^{(}->}[d]^{\psi_1}\ar@{^{(}->}[r]\ar@{}[rd]|{\bigcup} &
Z\ar@{^{(}->}[d]^\psi\ar[r]\ar@{}[rd]|{\coprod} & \calY\ar@{^{(}->}[d]\ar@{=}[r] & \calY\ar@{=}[r]\ar[d]^{\psi_{\bfY}} & \calX\ar[d]^{\psi_\bfX}\\
T\ar[r] & U_1\ar@{^{(}->}[r] & U\ar[r] & V\ar[r] & Y\ar[r] & X}
\end{equation}

\subsubsection{Comparison with the schematic case from \cite{temrz}}
Before proving the theorem, let us compare this plan with the parallel proof in \cite{temrz}. That proof runs as follows. If $\bfX$ is schematic then $\bfT$ is also schematic by \cite[Proposition~5.1.9]{pruf}. Consider the semivaluation ring $\calO$ of $\bff$ (see \cite[\S3.1]{temrz}), then $\Spec(\calO)$ is the affine Ferrand pushout $\calX_\eta\coprod_\eta T$ (see \cite[\S2.3]{temrz}). Approximate $\Spec(\calO)$ by an affine $X$-scheme $W$ of finite presentation, and note that for a small enough neighborhood $U\subseteq \calX$ of $\eta$ the morphism $U\to X$ lifts to a morphism $U\to W$ by approximation. Let $Y$ be the schematic image of $U$ in $W$. Then $Y$ is of finite type. Set $\bfY:=(U\to Y)$. Finally, care to choose $U$ and $W$ such that $\bfY\to \bfX$ is adic; this is the most subtle part.


Assume, now, that $\bfX$ is general. If the localization $X_\eta$ exists then the above proof applies: one considers the Ferrand pushout $\calX_\eta\coprod_\eta T$ and approximates it by an $X$-space of finite type to get a quasi-modification along $\bff$. However, localizations of algebraic spaces do not exist in general, and this forces us to use an open neighborhood of $\eta$ when forming a pushout. Then the pushout datum is not Ferrand, and we have to use few types of pushouts as outlined in \S\ref{plansec}.

\subsubsection{Step 1: the main ingredient}
Here is a key result used in the proof of \ref{strictqmodth}.

\begin{prop}\label{prop:affpushout}
Let $\bfX$ be a model, $\bfT\to\bfX$ an adic semivaluation and $Z\subseteq \calX$ the closure of $\calT$. Assume that $T$ is uniformizable (see \cite[\S5.2.19]{pruf}). Then,

(1) There exists a non-empty open affine subscheme $Z_1=\Spec(B)\into Z$ such that for any $0\neq h\in B$ the localization $Z_h=\Spec(B_h)$ satisfies the following properties: (a) $Z_h$ is $X$-affine, (b) $\psi_h\colon Z_h\to U_h:=Z_h\coprod_\calT^{X^\aff}T$ is a good affine model, and (c) the morphism $(U_h,Z_h,\psi_h)\to \bfX$ is adic.

(2) If $Z_1$ is as in (1) then  $\calT=\underleftarrow{\lim}_{0\neq h\in B}Z_h$ and $T=\underleftarrow{\lim}_{0\neq h\in B}U_h$. 
\end{prop}
Note that $T\to X$ is affine by \cite[Proposition~5.1.9(i)]{pruf}, and so the $X$-affine pushout $U_1$ makes sense. We illustrate the assertion of (1) with the following diagram.
$$\xymatrix{
\calT\ar@{^{(}->}[r]\ar@{^{(}->}[d]_{\psi_\bfT}\ar@{}[rd]|{\coprod^{X^\aff}} & Z_h\ar@{^{(}->}[r]\ar[d]^{\psi_h} & Z\ar@{^{(}->}[r] & \calX\ar[d]^{\psi_\bfX}\\
T\ar[r] & U_h\ar[rr] & & X\\
}$$
\begin{proof}
We will only consider non-zero localization, so the word ``non-zero" will be omitted. To ease the notation we start with any non-empty affine open subsheme $Z_1=\Spec(B)\into Z$, and replace it with its localizations few times till all conditions of (1) are satisfied. Throughout the proof we fix an affine presentation $\pi\colon \widetilde{X}\to X$. For any $X$-space $W$, denote the base change $W\times_{X}\widetilde{X}$ by $\widetilde{W}$. Similarly, for any element $a\in\calO(W)$ we denote its pullback by $\tila\in\calO(\tilW)$.

Step 1. {\it There exists $Z_1$ that satisfies (a) and hence any its localization satisfies (a).} Let $V\subseteq\tilZ_1$ be a dense open affine subscheme. Then there exists $a\in B$ such that $V$ contains $Z_a\times_X\tilX$, and hence $Z_a\times_X\tilX=V_{\tila}$ is affine. By descent, $Z_a\to X$ is an affine morphism, hence, after replacing $Z_1$ with $Z_a$, we may assume that $Z_1$ satisfies (a). Obviously, any further localization of $Z$ is also $X$-affine.

At this stage, many relevant spaces become affine, so it will be useful to consider the corresponding rings. Set $\tilB:=\calO(\tilZ_1)$, $K:=\Frac(B)=\calO(\calT)$, and $\tilK:=\calO(\tilcalT)$. The space $\tilT$ is a uniformizable SLP since $T$ is so and $\tilT\to T$ is \'etale. Furthermore, it is affine since $T\to X$ is so. Set $\tilR:=\calO(\tilT)$.

Step 2. {\it Shrinking $Z_1$ we can achieve that there exists $g\in B^\times\subset K$ whose image $\tilg\in\tilK$ is contained in $\tilR$ and satisfies $\tilR_\tilg=\tilK$.} By uniformizability of $\tilT$, there exists $\alp\in\tilR$ such that $\tilK=\tilR_\alp$, and hence there exists $g\in K$, e.g., the norm of $\alp$, such that $\tilg\in\tilR$ and $\tilR_\tilg=\tilK$. Since $K=\Frac(B)$, after replacing $Z_1$ with a localization we may assume that $g\in B^\times$, and hence $\tilg\in \widetilde{B}^\times\cap\tilR$.

Step 3. {\it Each localization $Z_h$ satisfies (b).} Let us show first that
\begin{equation}\label{eq:loc}
\left(\widetilde{B}\cap\tilR\right)_\tilg=\widetilde{B}.
\end{equation}
Since $\tilg^{-1}\in \widetilde{B}$, the direct inclusion follows. For the opposite one, note that for any $a\in \widetilde{B}\subseteq\tilK=\tilR_\tilg$ there exists $n$ such that $a\tilg^n\in\tilR$, and hence $a\tilg^n\in \widetilde{B}\cap\tilR$. Thus, the opposite inclusion also holds.

It follows from \eqref{eq:loc} that $\tilZ_1\to\tilZ_1\coprod_{\tilcalT}^{X^\aff}\tilT=\Spec\left(\tilB\cap\tilR\right)$ is an affine open immersion. And it is schematically dominant, since $\widetilde{\calT}$ is dense in $\Spec\left(\widetilde{B}\cap\tilR\right)$. By descent, $Z_1\to Z_1\coprod_\calT^{X^\aff}T$ is a good affine model. Furthermore, since $g\in B_h^\times$ for any $h$, the same argument shows that (b) is satisfied for any localization $Z_h$.

Step 4. {\it $Z_1$ satisfies the assertion of (2).} Note that $$\underrightarrow{\colim}_h(\tilB_h\cap\tilR)\toisom\tilR\, .$$ Thus, $\tilT\toisom\underleftarrow{\lim}_h\tilU_h$, and, by \'etale descent, the morphism $T\to\underleftarrow{\lim}_h U_h$ is an isomorphism. The first isomorphism of (2) follows easily.

Step 5. {\it End of proof: shrinking $Z_1$ we can achieve that any localization of $Z_1$ satisfies (c).} Condition (c) means that the diagonal
$Z_h\to U_h\times_X\calX$ is a closed immersion, hence it is satisfied if and only if the immersion $\phi\:Z_h\to U_h\times_XZ$ is closed. The latter factors through the closed immersion $\phi_h\:Z_h\to U_h\times_XZ_1$, so (c) is satisfied if and only if the closures of $Z_h$ in $U_h\times_XZ$ and $U_h\times_XZ_1$ coincide. Since $\calT$ is dense in $Z_h$, these closures coincide with the closures of $\calT$.

The immersion $i\:\calT\to T\times_XZ$ is closed because $\bfT\to\bfX$ is adic. Note that the image of $i$ is contained in $T\times_XZ_h$ for any $0\neq h\in B$. Since $|T\times_{X}Z|=\underleftarrow{\lim}_h|U_h\times_{X}Z|$ by Step 4, it follows by approximation that for any $h$ large enough (i.e., there exists $0\neq a\in B$ such that for any $0\neq h\in aB$) the closure of $\calT$ in $U_h\times_XZ$ is contained in $U_h\times_X Z_1$.

We claim that replacing $Z_1$ with $Z_a$ one completes the step. It suffices to show that $Z_h\to U_h\times_X Z_a$ is a closed immersion for any $0\neq h\in aB$, and we will show that already the immersion $Z_h\to U_h\times_X Z_1$ is closed. By descent, this is equivalent to $\alp_h\:\calO(\widetilde{U}_h)\otimes\widetilde{B}\to\widetilde{B}_h$ being surjective. Clearly, we can replace $h$ with any element of the form $hg^n$, where $n\in\bfZ$ and $g$ is the unit from Step 2. Since $\tilK=\tilR_\tilg$ any choice of sufficiently negative $n$ guarantees that $\tilh^{-1}\in\tilR$. But then $\tilh^{-1}\in \calO(\widetilde{U}_h)=\widetilde{B}_h\cap\tilR$, so $\alp_h$ is surjective.
\end{proof}

\begin{proof}[Proof of Theorem~\ref{strictqmodth}]
Prior to implementing the plan of \S\ref{plansec} we need to fix some terminology. Recall that we identified $\calT$ with $\eta\in\calX$.

Setup: Let $\alp\:\tilcalX\to\calX$ be an \'etale morphism such that $\tilcalX$ is affine and $\alp$ is strictly \'etale over $\eta\in\calX$. Note that $\alp$ is not assumed to be surjective. For any $\calX$-space $W$, denote the base change $W\times_{\calX}\tilcalX$ by $\tilW$. Let $Z\subseteq \calX$ be the closure of $\eta$. Then $\tilZ$ is affine, and $\tilZ\to Z$ is generically an isomorphism. Let $Z_0\into Z$ be an open affine subscheme such that $\tilZ_0\to Z_0$ is an isomorphism and $\tilZ_0$ is a localization of $\tilZ$.

Step 1. Note that $T$ is uniformizable by \cite[Lemma~5.2.20]{pruf}. By Proposition~\ref{prop:affpushout}, $Z_0$ contains an open $X$-affine subscheme $Z_1$ such that $\psi_1\colon Z_1\to U_1:=Z_1\coprod_{\calT}^{X^\aff}T$ is a good affine model and the morphism $(U_1,Z_1,\psi_1)\to \bfX$ is adic, and these two properties also hold for any further localization of $Z_1$. We will use the latter to refine $\calX$.

 Pick a localization $Z_h$ of $Z_0$ contained in $Z_1$; it is also a localization of $Z_1$. Then $\tilZ_h$ is a localization of $\tilZ_0$, hence also a localization of $\tilZ$. In particular, there exists a localization $\tilcalX_h$ of $\tilcalX$ whose restriction onto $\tilZ$ equals $\tilZ_h$. Replacing $Z_1$ with $Z_h$ and $\tilcalX$ with $\tilcalX_h$ we achieve, in addition, that $\tilZ=\tilZ_1\toisom Z_1$.

Step 2. Let $U:=Z\cup U_1$ be the Zariski gluing of $U_1$ and $Z$ along $Z_1$. Then $\psi\:Z\to U$ is a schematically dominant affine open immersion, since $Z_1\to U_1$ is so. We claim that $U\to X$ is separated. Indeed, consider the diagonal $$\Delta_{U/X}\:U\to U\times_{X}U=(Z\times_{X}Z)\cup (U_1\times_{X}U_1)\cup (Z\times_{X}U_1)\cup (U_1\times_{X}Z)$$ Its restrictions onto the first two constituents of the target are the diagonals $\Delta_{Z/X}$ and $\Delta_{U_1/X}$, which are closed immersions by $X$-separatedness of $Z$ and $U_1$. The restrictions onto the last two constituents are isomorphic to the morphism $Z_1\to U_1\times_{X}Z$ whose composition with the closed immersion $U_1\times_XZ\into U_1\times_X\calX$ is the diagonal of the adic morphism $(U_1,Z_1,\psi_1)\to \bfX$. Hence they are closed immersions too, and we obtain that $\Delta_{U/X}$ is a closed immersion as well.

Step 3. Consider the Ferrand pushout datum $(Z;U,\calX)$, and let us show first that it is effective. By \cite[Theorem~6.2.1]{push}, it is sufficient to find an \'etale affine covering of $(Z;U,\calX)$. Choose any affine covering $\calX'\to\calX$, and consider the induced affine covering $Z':=\calX'\times_{\calX}Z$ of $Z$. Since $Z_1=Z\times_\calX\tilcalX$ by Step 1, we obtain that $(Z'\coprod Z_1; Z'\coprod U_1, \calX'\coprod \tilcalX)$ is a desired affine covering of $(Z;U,\calX)$.

Set $V:=U\coprod_Z \calX$. Then, by the definition of Ferrand pushouts, $\calX\to V$ is affine, and, by \cite[Theorem~6.3.5]{push}, it is a schematically dominant open immersion since $Z\to U$ is so. Furthermore, $\calX\to X$ and $U\to X$ are separated hence $V\to X$ is separated by \cite[Theorem~6.4.1]{push}.

Step 4. By approximation, we can factor the morphism $V\to X$ through some $Y$ such that $V\to Y$ is affine and $Y\to X$ is separated and of finite type. Moreover, after replacing $Y$ with the schematic image of $V$, we may assume that $V\to Y$ is schematically dominant, and hence so is $\calX\to Y$. Set $\bfY:=(\calX\to Y)$. Then $\bfY$ is a strict pseudo-modification, hence a strict quasi-modification, along $\bfT$. By construction, $\bfY$ is affine.

Step 5. In fact, this step reduces to the claim that any affine finite type morphism of algebraic spaces is quasi-projective. It seems to be missing in the literature, so we provide a proof. By approximation (see \cite[Th. A]{rydapr}), the $\calO_Y$-module $(\psi_\bfY)_*\calO_\calY$ is a filtered colimit of its finitely generated $\calO_Y$-submodules $\calF_i$. We claim that the finitely generated $\calO_Y$-algebra $(\psi_\bfY)_*\calO_\calY$ is generated by some $\calF_i$. Indeed, it suffices to check this on an affine presentation, and then the assertion is obvious. Thus, $\psi_\bfY$ factors into a composition of a closed immersion $\alp\:\calY\to\bfP(\calF_i)$ and a projective morphism $\bfP(\calF_i)\to Y$, where $\bfP(\calF_i)$ is the projective fiber as defined in \cite[II, Def. 4.1.1]{ega}. Let $Y'$ be the schematic image of $\alp$, then $\calY\into Y'$ is a good affine quasi-modification of $\bfX$ along $\bfT$.
\end{proof}

\subsection{Modifications of finite type models}\label{cofsec}
In this section we will prove our main results on modifications of a finite type model $\bfX$.

\subsubsection{Gluing of quasi-modifications}\label{globalsec}
Theorem~\ref{strictqmodth} provides a local construction of good quasi-modifications of $\bfX$. In order to glue them to a single modification, one has first to modify them. In fact, as in Raynaud's theory, we will only use blow ups. The latter is crucial as we will use the extension property of blow ups.

\begin{lem}\label{lem:glqmod}
Let $\bfX$ be a model of finite type, and $\bff_i\colon\bfY_i\to \bfX$, $1\le i\le n$, strict good quasi-modifications. Then there exists a strict quasi-modification $\bfU\to\bfX$ and an open covering $\bfU=\cup_{i=1}^n \bfU_i$ such that $\bfU_i\to\bfY_i$ is a blow up for each $i$.
\end{lem}
\begin{proof}
Since blow ups are preserved by compositions, it is sufficient to prove the lemma for $n=2$. By Lemma~\ref{modlem} (1) and Lemma~\ref{normmodlem} (1), $\mathbf{Y}_{12}:=\mathbf{Y}_1\times_{\mathbf{X}}\mathbf{Y}_2$ is a good quasi-modification of $\bfX$. We are given the left diagram below, that we are going to extend to the right diagram.
$$
\xymatrix{
\calY_1\ar@{}[rd]|\square\ar@{^{(}->}[r]\ar@{=}[d]& Y_1\ar[r] & X\ar@{=}[d] & & \calY_1\ar@{=}[d]\ar@{^{(}->}[r] & U_1\ar@{->}[r] & Z_1\ar@{->}[r] & Y_1\ar[r] & X\ar@{=}[d]\\
\calY_{12}\ar@{}[rd]|\square\ar@{^{(}->}[r]\ar@{=}[d] & Y_{12}\ar[r]\ar[u]\ar[d] & X\ar@{=}[d] & & \calY_{12}\ar@{^{(}->}[r]\ar@{=}[d] & U_{12}\ar@{^{(}->}[u]\ar@{->}[r]\ar@{_{(}->}[d] & Z_{12}\ar@{^{(}->}[u]\ar@{->}[r]\ar@{->}[d] & Y_{12}\ar[u]\ar[r]\ar[d] & X \\
\calY_2\ar@{^{(}->}[r] & Y_2\ar[r] & X & & \calY_2\ar@{^{(}->}[r] & U_2\ar@{->}[r] & Z_2\ar@{=}[r] & Y_2\ar[r] & X\ar@{=}[u]\\}
$$
By Lemma~\ref{normmodlem} (2), the left upper square in the left diagram is cartesian. Thus, by \cite[Corollary 5.7.11]{RG}, there exists a $\calY_1$-admissible blow up $Z_1\to Y_1$ such that the strict transform $Z_{12}$ of $Y_{12}\to Y_1$ is flat and quasi-finite over $Z_1$. Since $Z_{12}\to Z_1$ is birational, it is an open immersion. Similarly, there exists a $\calY_2$-admissible blow up $U_2\to Z_2=Y_2$ such that the strict transform $U_{12}$ of $Z_{12}\to Z_2$ is flat and quasi-finite over $U_2$, hence $U_{12}\to U_2$ is an open immersion.

Note that $\bfU_2:=(\calY_2\to U_2)$ is a blow up of $\bfY_2$, and $\bfU_{12}:=(\calY_{12}\to U_{12})$ is both an open submodel of $\bfU_2$ and a blow up of $\bfZ_{12}:=(\calY_{12}\to Z_{12})$. By Proposition~\ref{prop:extofblups}, the blow up $\bfU_{12}\to\bfZ_{12}$ can be extended to a blow up $\bfU_1\to\bfZ_1$. Then $\bfU_1$ and $\bfU_2$ glue along $\bfU_{12}$ to a model $\bfU$, which satisfies the requirements of the lemma.
\end{proof}

\subsubsection{Nagata's compactification}
Now, we are in a position to prove Nagata's compactification theorem for algebraic spaces.

\begin{theor}\label{th:decompose1fintype}
Let $\bfX$ be a model of finite type. Then there exists an affine good modification $\bfX'\to\bfX$.
\end{theor}
\begin{proof}
By Theorem~\ref{strictqmodth}, for any adic semivaluation $\bfT\in\Val(\bfX)$ there exists a strict quasi-modification $\bfY_\bfT\to \bfX$ along $\bfT$ such that $\bfY_\bfT$ is affine and good. Since $\Val(\bfX)$ is quasi-compact by Proposition~\ref{compprop}, there exists a {\em finite} collection of such quasi-modifications $\bfY_i\to \bfX$, $1\le i\le n$, for which $\Val(\bfX)=\cup_{i=1}^n\Val(\bfY_i)$.

By Lemma~\ref{lem:glqmod}, there exists a strict quasi-modification $\bfX'\to\bfX$ and an open covering $\bfX'=\cup_{i=1}^n\bfY'_i$ such that $\bfY'_i\to \bfY_i$ is a blow up for each $i$. Then $\bfY'_i$ are affine and good since $\bfY_i$ are so. Thus, $\bfX'$ is affine and good too, and $\Val(\bfX)=\cup_{i=1}^n\Val(\bfY_i)=\cup_{i=1}^n\Val(\bfY'_i)=\Val(\bfX')$. Hence $\bfX'\to\bfX$ is a modification by Lemma~\ref{lem1}, and we are done.
\end{proof}

Using the flattening theorem of Raynaud-Gruson once again, we can slightly improve the theorem as follows.

\begin{cor}\label{cor:decomposeth}
Let $\bfX$ be a model of finite type, and $\bfY_i\to \bfX$, $1\le i\le n$, quasi-modifications. Then there exist affine good modifications $\bff\:\bfX'\to\bfX$ and $\bfg_i\:\bfY'_i\to\bfY_i$ such that $\bfY'_i\subseteq\bfX'$ are open submodels for all $i$. Furthermore, if $\bfX$ is good then the modifications $\bff$ and $\bfg_1\.\bfg_n$ can be chosen to be blow ups.
\end{cor}
\begin{proof}
It is sufficient to prove the corollary for $n=1$. Set $\bfY:=\bfY_1$. By Theorem~\ref{th:decompose1fintype}, there exists an affine good modification $\bfW\to \bfX$. Thus, after replacing $\bfX$ with $\bfW$ and $\bfY$ with $\bfY\times_\bfX\bfW$, we may assume that $\bfX$ is good and affine. Then $\bfY\to \bfX$ is semi-cartesian by Lemma~\ref{normmodlem} (2), and hence $\bfY$ is good and affine. By \cite[Corollary 5.7.11]{RG}, there exists an $\calX$-admissible blow up $f\:X'\to X$ such that the strict transform $Y'\to X'$ of $Y\to X$ is an open immersion. Thus, $\bfX'=(\calX\to X')$ and $\bfY':=(\calY\to Y')$ are as needed.
\end{proof}

\subsection{Elimination of finite type hypothesis}
In this section, we use the approximation theory to extend our theory of modifications to the case of general models.

\subsubsection{Approximation of models}\label{apprsec}
By \cite[Lemma~2.1.10]{pruf}, given a model $\bfX$ one can represent $\calX$ as a filtered limit of finite type $X$-spaces $\calX_\alp$ with affine and schematically dominant transition morphisms. Then each $\bfX_\alpha:=(\calX_\alpha\to X)$ is a model, and we obtain an approximation morphism $\bfh_\alp\:\bfX\to\bfX_\alp$ such that $h_\alp$ is an isomorphism and $\gth_\alp$ is affine and schematically dominant. Approximation of a couple of models connected by a morphism is a slightly more delicate issue, which we address in the following lemma.

\begin{lem}\label{approxmodlem}
Let $\bff\:\bfY\to\bfX$ be a morphism of models such that $Y\to X$ is of finite type and $\calY\to\calX$ is flat and of finite presentation (e.g., a pseudo-modification of models). Fix a family of models $\bfX_\alp$ as above. Then for a large enough $\alp$ there exists a morphism $\bff_\alp\:\bfY_\alp\to\bfX_\alp$ and an $\bfX$-isomorphism $\bfY\toisom\bfY_\alp\times_{\bfX_\alp}\bfX$ such that $Y=Y_\alp$ and $\calY_\alp\to\calX_\alp$ is flat and of finite presentation. In addition, if $\bff$ belongs to a class $\bfP$ then $\bff_\alp$ can be chosen from the same class, where $\bfP$ is one of the following: pseudo-modifications, quasi-modifications, strict quasi-modifications.
\end{lem}
The assertion of the lemma can be illustrated by the following diagram:
$$
\xymatrix{
\calY \ar[r] \ar[d]^{\gtf} \ar@/^1pc/[rr]^{\psi_\bfY} \ar@{}[dr] |\square & \calY_\alp \ar[r]_{\psi_{\bfY_\alp}}\ar[d]^{\gtf_\alp} & Y \ar[d]^{f=f_\alp}\\
\calX \ar[r] \ar@/_1pc/[rr]_{\psi_\bfX} & \calX_\alp \ar[r]^{\psi_{\bfY_\alp}} & X
}
$$
The morphism $\calX\to\calX_\alp$ is affine and schematically dominant by the definition of $\calX_\alp$'s, so $\calY\to\calY_\alp$ is affine and schematically dominant by \cite[Lemma~2.1.6]{pruf}.
\begin{proof}
By approximation (see \cite[Propositions B.2 and B.3]{rydapr}), for $\alp$ large enough there exists an algebraic space $\calY_\alp$ and a flat finitely presented morphism $\gtf_\alp\:\calY_\alp\to\calX_\alp$ with an $\calX_\alp$-isomorphism $\calY\toisom\calY_\alp\times_{\calX_\alp}\calX$. For any $\beta\ge\alp$ set $\calY_\beta:=\calY_\alp\times_{\calX_\alp}\calX_\beta$. Then $\calY=\underleftarrow{\lim}_{\beta\ge\alp}\calY_\beta$ and the transition morphisms $\calY_\gamma\to\calY_\beta$ are affine and schematically dominant by \cite[Lemma~2.1.6]{pruf}. Thus, for $\beta$ large enough the $X$-morphism $\calY\to Y$ factors through $\calY_\beta$ by \cite[Lemma~2.1.11]{pruf}, and after replacing $\alp$ with such $\beta$ we may assume that it factors already through $\calY_\alp$. Hence the model $\bfY_\alp=(\calY_\alp\to Y)$ and the morphism $\bfY_\alp\to\bfX_\alp$ are as needed.

It remains to address the last statement. Assume that $f$ is a quasi-modification. By approximation, $Y$ embeds as a closed subscheme into a finitely presented $X$-space $W$. Since $\Delta_\bff\:\calY\to Y\times_X\calX$ is a closed immersion, $\calY\to W\times_X\calX$ is a closed immersion of finitely presented $\calX$-spaces. By \cite[Proposition B.3]{rydapr}, for a large enough $\alp$ we have that $\calY_\alp\to W\times_X\calX_\alp$ is a closed immersion. Since it factors through a closed immersion $Y\times_X\calX_\alp\to W\times_X\calX_\alp$, we obtain that $\Delta_{f_\alp}\:\calY_\alp\to Y\times_X\calX_\alp$ is a closed immersion, and hence $f_\alp$ is a quasi-modification. Finally, if $\gtf\:\calY\to\calX$ is an open immersion or an isomorphism, then we can choose $\gtf_\alp$ to be so.
\end{proof}

\subsubsection{Main results on modifications}
Using approximation we can eliminate some unnecessary finite type assumptions. Here is our main result on modifications of models that contains, in particular, the earlier results on finite type models.

\begin{theor}\label{thm:modmainth}
Let $\bfX$ be a model and $\bff_i\:\bfY_i\to\bfX$, $1\le i\le n$, be quasi-modifications. Then there exist affine modifications $\bfh\:\bfX'\to\bfX$ and $\bfg_i\:\bfY'_i\to\bfY_i$ such that $\bfY'_i\subseteq\bfX'$ are open submodels for all $i$. Moreover, if $\bfX$ is good then the modifications $\bfh$ and $\bfg_1\.\bfg_n$ can be chosen to be blow ups.
\end{theor}
\begin{proof}
It is sufficient to prove the Theorem for $n=1$, so set $\bfY:=\bfY_1$ and $\bff:=\bff_1$. Note that the additional claim follows from Corollary~\ref{cor:decomposeth} since any good model is of finite type. Let us construct $\bfh$ and $\bfg=\bfg_1$ in the general case. By Lemma~\ref{approxmodlem}, there exists a quasi-modification $\bff_0\:\bfY_0\to\bfX_0$ of finite type models and a morphism $\bfX\to\bfX_0$ such that $\bff=\bff_0\times_{\bfX_0}\bfX$, $X=X_0$, $Y=Y_0$ and the morphisms  $\calX\to\calX_0$ and $\calY\to\calY_0$ are affine and schematically dominant. By Corollary~\ref{cor:decomposeth}, there exist modifications $\bfY'_0\to\bfY_0$ and $\bfX'_0\to\bfX_0$ with affine sources and such that $\bfY'_0\subseteq\bfX'_0$ is an open submodel.

The pullbacks $\bfY'=\bfY'_0\times_{\bfY_0}\bfY$ and $\bfX'=\bfX'_0\times_{\bfX_0}\bfX$ are affine modifications of $\bfY$ and $\bfX$, respectively. In particular, the natural morphism $\bfY'\to\bfX'$ is a quasi-modification. Furthermore, $Y'$ is the schematic image of $\calY'_0\times_{\calY_0}\calY=\calY$ in $Y'_0\times_{Y_0}Y=Y'_0$. Since $\calY\to\calY_0$ and $\calY_0=\calY'_0\to Y'_0$ are schematically dominant, we have that $Y'=Y'_0$. So, the morphism $Y'\to X'$ is nothing but the open immersion $Y'_0\into X'_0=X'$, and we obtain that $\bfY'\to\bfX'$ is, actually, an open immersion.
\end{proof}

\section{Main results on Riemann-Zariski spaces}\label{rzsec}

\subsection{The family of quasi-modifications along a semivaluation}\label{qmodsec}
In \S\ref{strictqmodsec}, we studied the family of strict quasi-modifications of a model $\bfX$ along an adic semivaluation $\bfT$. In this section we extend the study to the non-strict case.

\subsubsection{Quasi-modifications of $\bfX$ and the topology of $\calX$}
The following proposition is the main new ingredient we need to describe the family of all quasi-modifications of a model $\bfX$ along a semivaluation $\bfT$.

\begin{prop}\label{qmodlem}
Assume that $\bfX$ is a model with a semivaluation $\bfv\:\bfT\to\bfX$, and $U$ is a neighborhood of $\calT$ in $\calX$. Then there exists a quasi-modification $\bfY\to\bfX$ along $\bfT$ such that $\calY\subseteq U$. Moreover, if $\bfX$ is good then $\bfY$ can be chosen to be a quasi-blow up of $\bfX$.
\end{prop}
\begin{proof}
First, let us reduce to the case when the model is of finite type. For this we realize $\bfX$ as the limit of finite type models $\bfX_\alp$ with affine transition morphisms, see \S\ref{apprsec}. By approximation, for large enough $\alp$ there exists an open subspace $U_\alp\into\calX_\alp$ whose preimage in $\calX$ is $U$. Let $\bfv_\alp\:\bfT_\alp\to\bfX_\alp$ be the image of $\bfv$ under the map $\phi_\alp\:\Spa(\bfX)\to\Spa(\bfX_\alp)$. If there exists a quasi-modification $\bfY_\alp\to\bfX_\alp$ along $\bfv_\alp$ such that $\calY_\alp\subseteq U_\alp$ then its pullback $\bfY\to\bfX$ is as required. Thus, we may assume that $\bfX$ is of finite type. Note that even if we have started with an adic $\bfv$, the adicity property can be lost at this stage since the morphism $\bfX\to\bfX_\alp$ is not necessarily adic.

By Lemma~\ref{qmodrem}(2), after replacing $\bfv$ with $r_\bfX(\bfv)$ we may assume that $\bfv$ is adic. Next, let us reduce to the case of a good model. By Theorem~\ref{thm:modmainth}, there exists a good modification $\bfX'\to\bfX$. Let $\bfT\to\bfX'$ denote the lifting of $\bfv$. If it factors through a quasi-modification $\bfY\to\bfX'$ with $\calY\subseteq U$, then the composition $\bfY\to\bfX$ is as required. Thus, we may assume that $\bfX$ is good, and our aim is to find a quasi-blow up $\bfY\to\bfX$ as asserted by the lemma.

Let $W$ be the closure of $\calX\setminus U$ in $X$. It is sufficient to find an $\calX$-admissible blow up $h\:\oY\to X$ that {\em separates} $W$ from $T$ in the sense that the strict transform $W'$ of $W$ is disjoint from the image of the lifting $T\to \oY$ of $v$. Indeed, then we can take $Y:=\oY\setminus W'$ and $\calY:=Y\cap \calX$. Existence of such $h$ is proved in the following lemma.
\end{proof}

\begin{lem}
Assume that $\bfX$ is a good model, $\bfT\to\bfX$ an adic semivaluation, and $W$ a closed subset of $|X|$ not containing the generic point $\calT$ of $\bfT$. Then there exists an $\calX$-admissible blow $\oY\to X$ that separates $W$ from $T$.
\end{lem}
\begin{proof}
Enlarging $W$ we can achieve that $X\setminus W$ is quasi-compact and still $\calT\notin W$. Since $\calX$ is quasi-compact, there exists an ideal $\calI\subseteq\calO_X$ of finite type such that $V(\calI)=X\setminus \calX$. Similarly, there exists an ideal $\calJ\subseteq\calO_X$ of finite type such that $W=V(\calJ)$. Set $\calJ_m:=\calJ+\calI^m$ and $\oY_m:=Bl_{\calJ_m}(X)$. To prove the lemma, we will establish a stronger claim as follows: {\it the $\calX$-admissible blow up $g_m\:\oY_m\to X$ separates $W$ from $T$ for any large enough $m$}.

Assume, first, that $X=\Spec(A)$ is affine and $T=\Spec(R)$ is a valuation scheme. Let $I=(f_1\.f_r)$ and $J=(g_1\.g_k)$ be the ideals of $A$ corresponding to $\calI$ and $\calJ$. Let $\alp_1\.\alp_s$ be the generators of $I^m$ of the form $\prod_{j=1}^mf_{i_j}$, and let $\of_i$, $\og_j$, and $\oalpha_l$ denote the images of $f_i$, $g_j$, and $\alp_l$ in $R$. Since $R$ is a valuation ring, there exists $l$ such that $\og_l$ divides $\og_j$ for any $1\le j\le k$. We may assume that $l=1$. We claim that there exists $n_0$ such that $\of_j^n$ is divisible by $\og_1$ for any $1\le j\le k$ and any $n\ge n_0$. Indeed, since $\bfT\to\bfX$ is adic, $\of_j$ vanishes on the complement of the generic point, and hence $\frac{1}{\og_1}\in\Frac(R)=R_{\of_j}$, which implies the claim. Thus, if $m\ge rn_0$ then $\og_1$ divides all elements $\oalpha_l$. Hence the map $T\to X$ lifts to the $g_1$-chart $$\Spec \left(A\left[\frac{g_2}{g_1}\.\frac{g_k}{g_1},\frac{\alpha_1}{g_1}\.\frac{\alpha_s}{g_1}\right]\right)$$ of the blow up along $J_m=J+I^m=(g_1\.g_k,\alp_1\.\alp_s)$. Since $W$ is contained in $V(g_1)$, its strict transform is disjoint from the $g_1$-chart. Thus, $J_m$ is as claimed.

Consider, now, the general case. Choose an affine \'etale presentation $\widetilde{X}\to X$ and recall that $\widetilde{T}:=T\times_X\widetilde{X}$ is an SLP space. Moreover, $\tilT$ is separated over $\tilX$, and hence is a scheme by \cite[Proposition~5.1.9(i)]{pruf}. Set $\tilcalI=\calI\calO_\tilX$ and $\tilcalJ=\calJ\calO_\tilX$ and note that $\tilcalJ_m:=\calJ_m\calO_\tilX$ equals to $\tilcalJ+\tilcalI^m$. If $m$ is large enough then blowing up $\tilX$ along $\tilcalJ_m$ separates the strict transform of $\tilW=W\times_X\tilX$ from $\tilT$. Indeed, separating $W$ from $\tilT$ is equivalent to separating it from all localizations of $\tilT$ at the closed points. The latter are valuation schemes and there are finitely many of them, so our claim follows from the affine case that have already been verified. It remains to note that blow ups and strict transforms are compatible with flat morphisms, hence blowing up along $\calJ_m$ separates the strict transform of $W$ from $T$.
\end{proof}

\subsubsection{The main result}
Now, we are in a position to prove our main result about the family of quasi-modifications along an adic semi-valuation.

\begin{theor}\label{thm:qmodth}
Let $\bfX$ be a model, $\bfv\in\Val(\bfX)$, and $\bff\:\bfY\to\bfX$ a pseudo-modification along $\bfv$. Then there exists a quasi-modification $\bfg\:\bfZ\to\bfX$ along $\bfv$ such that $\bfZ$ is affine and $\bfg$ factors through $\bff$. Furthermore, $\bfg$ and $\bfZ$ can be chosen in such a way that:

(i) if $\bff$ is strict then so is $\bfg$,

(ii) if $\bfX$ is good then $\bfg$ is a quasi-blow up,

(iii) if $\bfX$ is of finite type then $\bfZ$ is good.
\end{theor}
\begin{proof}
Case 1. {\it $\bff$ is strict.} In this case $\bff$ is a quasi-modification by Proposition~\ref{prop:strpmISqm}, and the assertion of the theorem follows from Theorems \ref{th:decompose1fintype} and \ref{thm:modmainth}. 

Case 2. {\it $\bfX$ is good and affine.} By Proposition~\ref{qmodlem}, there exists a quasi-blow up $\bfX'\to\bfX$ along $\bfv$ such that $\calX'\into \calY\into \calX$. Then $\bfY'=\bfX'\times_\bfX\bfY$ is a strict pseudo-modification of $\bfX'$ along $\bfv$, and hence a quasi-modification by Proposition~\ref{prop:strpmISqm}. By Case 1, there exists a quasi-blow up $\bfZ\to\bfX'$ along $\bfv$ which factors through $\bfY'$. It remains to note that the composition $\bfZ\to\bfX'\to\bfX$ is a quasi-blow up by Proposition~\ref{prop:compqblups}.

Case 3. {\it $\bfX$ is of finite type.} By Theorem~\ref{th:decompose1fintype}, there exists a good affine modification $\overline{\bfX}\to\bfX$. Then $\obfY=\bfY\times_\bfX\overline{\bfX}$ is a pseudo-modification of $\obfX$, and by Case~2, there exists a quasi-blow up $\bfZ\to\obfX$ along $\bfv$ that factors through $\obfY$. The composition $\bfZ\to\obfX\to\bfX$ is a required quasi-modification of $\bfX$ by Lemma~\ref{normmodlem}. In particular, if $\bfX$ is good, then the modification $\overline{\bfX}\to\bfX$ can be chosen to be a blow up, and hence, in this case, $\bfZ\to\obfX\to\bfX$ is a quasi-blow up by Proposition~\ref{prop:compqblups}.

Case 4. {\it $\bfX$ is general.} First, by picking an affine modification $\obfX\to\bfX$ and proceeding as in Case 3, we reduce to the case when $\bfX$ is affine. Then, by Proposition~\ref{qmodlem}, there exists a quasi-modification $\bfX'\to\bfX$ along $\bfv$ such that $\calX'\into \calY\into \calX$. So, the argument from Case 2 reduces to Case 1.
\end{proof}

\subsection{Applications to RZ spaces}

\subsubsection{Topology on $\Val(\bfX)$}
Theorem~\ref{thm:qmodth} implies that the topology on the spaces $\Val(\bfX)$ is generated by quasi-modifications. In fact, a stronger statement is true:

\begin{theor}\label{topvalth}
Let $\bfX$ be a model, and $U\subseteq\Val(\bfX)$ be a quasi-compact open subset. Then there exists a quasi-modification $\bfW\to \bfX$ such that the image of $\Val(\bfW)$ in $\Val(\bfX)$ coincides with $U$.
\end{theor}
\begin{proof}
By definition, there exists a finite covering $U=\cup_{i=1}^r\left(\Spa(\bfY_i)\cap\Val(\bfX)\right)$, where $\bfY_i\to \bfX$ are pseudo-modifications. By Theorem~\ref{thm:qmodth}, for any semivaluation $\bfT\in \Spa(\bfY_i)\cap\Val(\bfX)$ there exists a quasi-modification $\bfZ\to \bfX$ along $\bfT$ which factors through $\bfY_i$. Obviously, $\Val(\bfZ)$ is a neighborhood of $\bfT$ in $\Spa(\bfY_i)\cap\Val(\bfX)$. Thus, $U=\cup_{i=1}^n\Val(\bfZ_i)$ for a finite collection of quasi-modifications $\bfZ_i\to\bfX$. By Theorem~\ref{thm:modmainth}, there exist modifications $\bfX'\to \bfX$ and $\bfZ'_i\to \bfZ_i$ for $1\le i\le n$ such that $\bfZ'_i$ are open submodels of $\bfX'$. Thus, $\bfW:=\cup_{i=1}^n\bfZ'_i\subseteq\bfX'$ is as needed.
\end{proof}

\subsubsection{Homeomorphism of $\Val(\bfX)$ and $\RZ(\bfX)$}
We saw in \S\ref{mapstorzsec} that the reduction maps induce a map $\red_\bfX\:\Val(\bfX)\to\RZ(\bfX)$, and we are going to prove that it is a homeomorphism. Actually, the latter is essentially equivalent to the combination of Theorems~\ref{topvalth} and \ref{thm:modmainth}.

\begin{theor}\label{rzth}
The map $\red_\bfX\:\Val(\bfX)\to\RZ(\bfX)$ is a homeomorphism.
\end{theor}
\begin{proof}
Recall that $\red_\bfX$ is continuous by Proposition~\ref{prop:cont}, and is surjective by Corollary~\ref{cor:surjofred}. If $U\subset\RZ(\bfX)$ is open and quasi-compact then the open set $\red_\bfX^{-1}(U)$ is quasi-compact by Proposition~\ref{compprop}. In particular, we obtain a map $\red_\bfX^{-1}$ between quasi-compact open subsets of $\RZ(\bfX)$ and $\Val(\bfX)$. This map is injective by the surjectivity of $\red_\bfX$, and we claim that, in fact, it is a bijection.

Let $V\subset\Val(\bfX)$ be a quasi-compact open subset. By Theorem~\ref{topvalth}, there exists a quasi-modification $\bfW\to\bfX$ such that the image of $\Val(\bfW)$ in $\Val(\bfX)$ is $V$. By Theorem~\ref{thm:modmainth}, there exist modifications $\bfW'\to\bfW$ and $\bfX'\to\bfX$ such that $\bfW'$ is an open submodel in $\bfX'$. Recall that $\Val(\bfW)=\Val(\bfW')$ and $\Val(\bfX)=\Val(\bfX')$ by Lemma~\ref{lem1}. Since $\Val(\bfW')\subset\Val(\bfX')$ is the preimage of the quasi-compact open subset $W'\subset X'$, it is also the preimage of the quasi-compact open subset of $\RZ(\bfX')=\RZ(\bfX)$ defined by $W'$.

The topologies on both spaces are generated by open quasi-compact sets, hence it remains to show that $\red_\bfX$ is injective. The latter is true since the topology on $\Val(\bfX)$ is induced from $\RZ(\bfX)$ by what we have proved above and distinguishes points by Theorem~\ref{thm:ValT0}.
\end{proof}

We conclude this section with the following corollary that generalizes certain results of \cite{temrz} on schematic models.

\begin{cor}\label{rzcor}
Let $\bfX$ be a model. Then,

(1) The natural map $i_\bfX\:\calX\to\RZ(\bfX)$ is injective, and any point $x\in\RZ(\bfX)$ possesses a unique minimal generalization in $i_\bfX(\calX)$.

(2) The retraction $r_\bfX\:\Spa(\bfX)\to\Val(\bfX)$ is a topological quotient map.
\end{cor}
\begin{proof}
We have a natural embedding $j_\bfX\:\calX\into\Val(\bfX)$ that sends $y\in\calX$ to the trivial adic semivaluation $\bfy:=(y\toisom y)\to\bfX$. Since $i_\bfX=\red_\bfX\circ j_\bfX$, it is sufficient to show that for an adic semivaluation $\bfT\to \bfX$ with $\calT=y$, the semivaluation $\bfy$ is its minimal generalization in $j_\bfX(\calX)$. Plainly, $\bfy$ is a generalization of $\bfT$. Let $y'\in\calX$ be a point whose closure does not contain $y$. Choose any open neighborhood $y\in \calU$ not containing $y'$, and set $\bfU:=(\calU\to X)$. Then $\bfU\to\bfX$ is a pseudo-modification along $\bfT$, and by Theorem~\ref{thm:qmodth}, there exists a quasi-modification $\bfZ\to\bfX$ along $\bfT$ that factors through $\bfU$. Obviously, $\Val(\bfZ)$ is a neighborhood of $\bfT$ that does not contain $j_\bfX(y')$. This proves (1).

Let $U\subseteq\Val(\bfX)$ be a subset. If $U$ is open and quasi-compact then, by Theorem~\ref{topvalth}, there exists a quasi-modification $\bfW\to \bfX$ such that $U=\Val(\bfW)$. Thus, $r_\bfX^{-1}(U)=\Spa(\bfW)$ is open in $\Spa(\bfX)$. Since the topology is generated by open quasi-compact subsets, we obtain that $r_\bfX$ is continuous. It remains to note that if $r_\bfX^{-1}(U)$ is open in $\Spa(\bfX)$ then $U=\Val(\bfX)\cap r_\bfX^{-1}(U)$ is open in $\Val(\bfX)$, and so $r_\bfX$ is a topological quotient map.
\end{proof}

\subsection{Miscellaneous}

\subsubsection{Limit of quasi-modifications}
An important part of our work was to study the filtered families of modifications of a model $\bfX$, and of quasi-modifications of $\bfX$ along an adic semivaluation $\bfT\to\bfX$. It is natural to ask if these families possess a limit. If $\bfX$ is schematic then the first family usually does not have a limit in the category of schemes, but admits a limit in the larger category of locally ringed spaces - the Riemann-Zariski space $\gtX=\RZ(\bfX)$. However, the second family possess a scheme limit, which is the spectrum of the semivaluation ring $\calO_{\gtX,\bfT}$ of $\bfT$. For general algebraic space models, it may happen that the limit of the second family does not exist, but one can identify the obstacle for this -- non-existence of the localization $\calX_\calT$. We have the following very explicit description of the limit when the localization exists:

\begin{theor}
Let $\bfX$ be a model with a semivaluation $\bfv\:\bfT\to\bfX$, and let $\{\bfX_\alp\to\bfX\}_\alp$ (resp. $\{\bfX'_\beta\to\bfX\}_\beta$) be the family of all pseudo-modifications (resp. quasi-modifications) along $\bfT$. Assume that the localization $\calX_\calT$ exists. Then,

(1) The Ferrand pushout $Z=\calX_\calT\coprod_\calT T$ and the limit $\underleftarrow{\lim}_\alp X_\alp$ exist, and the natural map $Z\to\underleftarrow{\lim}_\alp X_\alp$ is an isomorphism.

(2) Assume, in addition, that $\bfv$ is adic. Then $Z\toisom\underleftarrow{\lim}_\beta X'_\beta$. Moreover, if $\bfX$ is good then $Z$ is also isomorphic to the limit of all quasi-blow ups along $\bfT$.
\end{theor}
\begin{proof}
First, the Ferrand pushout exists by \cite[Theorem~6.2.1(ii)]{push}. Furthermore, for any $\alp$ the morphisms $T\to X_\alp$ and $\calX_\calT\to \calX_\alp\to X_\alp$ induce a morphism $Z\to X_\alp$, so a natural map of functors $Z\to\underleftarrow{\lim}_\alp X_\alp$ arises. (At this stage we do not know that the limit is representable.)

Next we claim that (2) follows from (1). Indeed, if $\bfv$ is adic then quasi-modifica-tions along $\bfT$ are cofinal among all pseudo-modifications along $\bfT$ by Theorem~\ref{thm:qmodth}. Moreover, if $\bfX$ is good then already quasi-blow ups along $\bfT$ form a cofinal subfamily by Theorem~\ref{thm:qmodth}(ii).

To prove (1) we will first construct a filtered family of pseudo-modifications $\{\bfU_\gamma\to\bfX\}_\gamma$ along $\bfT$ such that the transition morphisms $U_{\gamma'}\to U_\gamma$ are affine, $Z=\underleftarrow{\lim}_\gamma U_\gamma$ and $\calX_\calT=\cap_\gamma\calU_\gamma$. Note that $Z$ is $X$-separated by \cite[Theorem~6.4.1]{push}. By approximation (see \cite[Theorem D]{rydapr}), $Z$ is isomorphic to the filtered limit of finitely presented $X$-separated $X$-spaces $Z_\delta$ with affine transition morphisms. By \cite[Theorem~A.2.1]{pruf}, each morphism $\calX_\calT\to Z_\delta$ factors through a sufficiently small neighborhood $W_\delta$ of $\calT$ in $\calX$. Define now our family of pseudo-modifications as follows: let $\calU_\gamma$ be an open neighborhood of $\calT$ in some $W_\delta$ and let $U_\gamma$ be the schematic image of $\calU_\gamma$ under the morphism $W_\delta\to Z_\delta$. It is easy to see that this family satisfies all required properties.

Finally, we claim that the subfamily $\{\bfU_\gamma\}$ of $\{\bfX_\alp\}$ is cofinal; once we prove this the theorem follows. Fix any $\alp$. By approximation, $X_\alp$ embeds as a closed subspace into an $X$-space $Y$ of finite presentation. Consider the morphism $Z\to Y$ induced by $Z\to X_\alp$. By approximation, it factors through some $U_\gamma$. Furthermore, enlarging $\gamma$ we can also achieve that $\calU_\gamma\subseteq\calX_\alp$. Since the morphism $\calU_\gamma\to U_\gamma$ is schematically dominant, and the composition $\calU_\gamma\to U_\gamma\to Y$ factors through the closed subspace $X_\alp\subseteq Y$ via the morphism $\calU_\gamma\to\calX_\alp\to X_\alp$, the morphism $U_\gamma\to Y$ also factors through $X_\alp$. Thus, the pseudo-modification $\bfU_\gamma\to\bfX$ factors through $\bfX_\alp\to\bfX$, and we are done.
\end{proof}

\subsubsection{Schematization}\label{schemeprsec}
Semivaluation pushouts as above are also useful in proving the following valuative criterion for schematization of algebraic spaces.

\begin{theor}\label{schemeth}
Let $X$ be a qcqs algebraic space with a schematically dense quasi-compact open {\em subscheme} $\calX$, and let $\bfX$ denote the corresponding good model. Then the following conditions are equivalent:

(i) There exists an $\calX$-admissible blow up $X'\to X$ such that $X'$ is a scheme.

(ii) For any semivaluation $\bfT\to\bfX$ the valuation space $T$ is separated.

(iii) For any adic semivaluation $\bfT\to\bfX$ the valuation space $T$ is separated.
\end{theor}

\begin{proof}
To prove (i)$\Rightarrow$(ii), let $X'\to X$ be an $\calX$-admissible blow up such that $X'$ is a scheme, and $\bfT\to\bfX$ a semivaluation. Since $X'\to X$ is proper, we can lift the map $T\to X$ to $T\to X'$, which is necessarily separated since so is $T\to X$. Thus, $T$ is an affine scheme by \cite[Proposition~5.1.9(ii)]{pruf}.

The implication (ii)$\Rightarrow$(iii) is obvious, so assume that (iii) is satisfied and let us construct $X'\to X$ as in (i).

First, we claim that any adic semivaluation $\bfv\:\bfT\to\bfX$ factors through a quasi-modification $\bfZ\to\bfX$ such that $Z$ is a scheme. Recall that $T$ is affine by \cite[Proposition~5.1.9(ii)]{pruf}. Consider the localization $\calX_\calT:=\Spec(\calO_{\calX,\calT})$ and let $W:=\calX_\calT\coprod^\aff_\calT T$ be the affine pushout. By \cite[Theorem~4.2.1]{push}, it is also the Ferrand pushout, and in particular, there exists a natural morphism $W\to X$. By approximation, \cite[Theorem D (iii)]{rydapr}, we can factor it through a finitely presented morphism $W_0\to X$ such that $W_0$ is affine. The morphism $\calX_\calT\to W_0$ factors through an open neighborhood $V$ of $\calT$ in $X$ by \cite[Theorem~A.2.1]{pruf}. Set $\calY:=V\times_X\calX=V\cap\calX$, and let $Y$ be the schematic image of $\calY\to W_0$. Then $Y$ is an affine scheme and $\bfY\to\bfX$ is a pseudo-modification along $\bfv$. So, by Theorem~\ref{thm:qmodth}, there exists a quasi-blow up $\bfZ\to\bfX$ that factors through $\bfY$. Furthermore, $Z\to Y$ is quasi-projective since $Z\to X$ is so, and hence $Z$ is a scheme.

Since $\Val(\bfX)$ is quasi-compact, we can find quasi-blow ups $\bfZ_1\.\bfZ_n$ of $\bfX$ such that each $Z_i$ is a scheme and $\Val(\bfX)=\cup_{i=1}^n\Val(\bfZ_i)$. By Theorem~\ref{thm:modmainth}, there exist blow ups $\bfX'\to\bfX$ and $\bfZ'_i\to\bfZ_i$ with open immersions $\bfZ'_i\into\bfX'$, and by Lemma~\ref{lem1}, one has that $\bfX'=\cup_i\bfZ'_i$. Thus, $X'\to X$ is an $\calX$-admissible blow up, and $X'$ is a scheme because it is covered by blow ups of $Z_i$'s.
\end{proof}

\subsubsection{Pr\"ufer models}
A model $\bfX$ is called {\em Pr\"ufer} if it does not admit non-trivial modifications.

\begin{theor}\label{prufmorth}
Let $\bfX$ be a model of algebraic spaces. The following conditions are equivalent:

(i) $\psi_\bfX$ is a pro-open immersion and $(X,\calX)$ is a Pr\"ufer pair.

(ii) Any quasi-modification $\bfX'\to \bfX$ is an open immersion.

(iii) $\bfX$ is Pr\"ufer.
\end{theor}
\begin{proof}
The implication (i)$\Rightarrow$(ii) is obvious, and the equivalence (ii)$\Leftrightarrow$(iii) follows from Theorem~\ref{thm:modmainth}. Let us prove the implication (ii)$\Rightarrow$(i). By Theorem~\ref{thm:modmainth}, $\psi_\bfX$ is affine. Thus, by \cite[Lemma~2.1.10]{pruf}, $\calX$ is the filtered limit of $X$-spaces $X_\alp$ of finite type such that the transition morphisms are affine and schematically dominant. Then $\bfX_\alpha:=(\calX\to X_\alpha)$ is a strict quasi-modification of $\bfX$. By (ii), $\bfX_\alpha\to \bfX$ is an open immersion. Thus, $\psi_\bfX\:\calX\to X$ is the filtered limit of open immersions $X_\alp\into X$, and hence a pro-open immersion. Condition (i) now follows from \cite[Theorem~4.1.11]{pruf}.
\end{proof}

\begin{cor}\label{cor:prufcor}
(1) The class of Pr\"ufer morphisms between qcqs algebraic spaces is stable under \'etale base changes.

(2) A separated morphism $Y\to X$ between qcqs algebraic spaces is Pr\"ufer if and only if for some (and hence any) presentation of $X$ the base change is Pr\"ufer.
\end{cor}
\begin{proof}
Theorem~\ref{prufmorth} reduces the Corollary to analogues statement about Pr\"ufer pairs, that was proved in \cite[Corollary~4.1.17]{pruf}.
\end{proof}

\bibliographystyle{amsalpha}
\bibliography{nagata}

\providecommand{\bysame}{\leavevmode\hbox to3em{\hrulefill}\thinspace}
\providecommand{\MR}{\relax\ifhmode\unskip\space\fi MR }
\providecommand{\MRhref}[2]{%
  \href{http://www.ams.org/mathscinet-getitem?mr=#1}{#2}
}
\providecommand{\href}[2]{#2}
\begin{thebibliography}{CLO12}

\bibitem[BL93]{BL}
Siegfried Bosch and Werner L{\"u}tkebohmert, \emph{Formal and rigid geometry.
  {I}. {R}igid spaces}, Math. Ann. \textbf{295} (1993), no.~2, 291--317.
  \MR{1202394 (94a:11090)}

\bibitem[Bre13]{Br}
Uri Brezner, \emph{{Birational spaces}}, ArXiv e-prints (2013),
  \url{http://arxiv.org/abs/1311.6919}.

\bibitem[CLO12]{CLO}
Brian Conrad, Max Lieblich, and Martin Olsson, \emph{Nagata compactification
  for algebraic spaces}, J. Inst. Math. Jussieu \textbf{11} (2012), no.~4,
  747--814. \MR{2979821}

\bibitem[Con]{conrad}
Brian Conrad, \emph{Universal property of non-archimedean analytification},
  \url{http://math.stanford.edu/~conrad/papers/univan.pdf}.

\bibitem[Con07]{con}
\bysame, \emph{Deligne's notes on {N}agata compactifications}, J. Ramanujan
  Math. Soc. \textbf{22} (2007), no.~3, 205--257. \MR{2356346 (2009d:14002)}

\bibitem[Gro67]{ega}
A.~Grothendieck, \emph{\'{E}l\'ements de g\'eom\'etrie alg\'ebrique.}, Inst.
  Hautes \'Etudes Sci. Publ. Math. (1960-1967).

\bibitem[Hub93]{Hub1}
R.~Huber, \emph{Continuous valuations}, Math. Z. \textbf{212} (1993), no.~3,
  455--477. \MR{1207303 (94e:13041)}

\bibitem[L{\"u}t93]{Lut}
W.~L{\"u}tkebohmert, \emph{On compactification of schemes}, Manuscripta Math.
  \textbf{80} (1993), no.~1, 95--111. \MR{1226600 (94h:14004)}

\bibitem[Nag63]{Nagata}
Masayoshi Nagata, \emph{A generalization of the imbedding problem of an
  abstract variety in a complete variety}, J. Math. Kyoto Univ. \textbf{3}
  (1963), 89--102. \MR{0158892 (28 \#2114)}

\bibitem[RG71]{RG}
Michel Raynaud and Laurent Gruson, \emph{Crit\`eres de platitude et de
  projectivit\'e. {T}echniques de ``platification'' d'un module}, Invent. Math.
  \textbf{13} (1971), 1--89. \MR{0308104 (46 \#7219)}

\bibitem[Ryd]{Rydh}
David Rydh, \emph{Compactification of tame deligne-mumford stacks},
  \url{http://www.math.kth.se/~dary/tamecompactification20110517.pdf}.

\bibitem[Ryd13]{rydapr}
Daviv Rydh, \emph{{Noetherian approximation of algebraic spaces and stacks}},
  ArXiv e-prints (2013), \url{http://arxiv.org/abs/0904.0227}.

\bibitem[{Sta}]{stacks}
The {Stacks Project Authors}, \emph{{\itshape Stacks Project}},
  \url{http://stacks.math.columbia.edu}.

\bibitem[Tem00]{temlocal}
Michael Temkin, \emph{On local properties of non-{A}rchimedean analytic
  spaces}, Math. Ann. \textbf{318} (2000), no.~3, 585--607. \MR{1800770
  (2001m:14037)}

\bibitem[Tem10]{temst}
\bysame, \emph{Stable modification of relative curves}, J. Algebraic Geom.
  \textbf{19} (2010), no.~4, 603--677. \MR{2669727 (2011j:14064)}

\bibitem[Tem11]{temrz}
\bysame, \emph{Relative {R}iemann-{Z}ariski spaces}, Israel J. Math.
  \textbf{185} (2011), 1--42. \MR{2837126}

\bibitem[TT13a]{push}
Michael Temkin and Ilya Tyomkin, \emph{{Ferrand's pushouts for algebraic
  spaces}}, ArXiv e-prints (2013), \url{http://arxiv.org/abs/1305.6014}.

\bibitem[TT13b]{pruf}
\bysame, \emph{{Pr\"ufer algebraic spaces}}, ArXiv e-prints (2013),
  \url{http://arxiv.org/abs/1101.3199}.

\bibitem[Zar40]{Zar}
Oscar Zariski, \emph{Local uniformization on algebraic varieties}, Ann. of
  Math. (2) \textbf{41} (1940), 852--896. \MR{0002864 (2,124a)}

\end{thebibliography}

\end{document}